\documentclass[preprint, english]{elsarticle}
\usepackage[utf8]{inputenc}
\usepackage{geometry}
\usepackage{amsthm}
\usepackage{amsmath}
\usepackage{amssymb}
\usepackage{xfrac}
\usepackage{graphicx}
\usepackage{enumerate}

\renewcommand{\Re}{\operatorname{Re}}
\renewcommand{\Im}{\operatorname{Im}}

\makeatletter

\numberwithin{figure}{section}
\theoremstyle{plain}

\theoremstyle{plain}
\newtheorem{thm}{Theorem}
\newtheorem{definition}{Definition}

\@ifundefined{date}{}{\date{}}

\newtheorem{lm}{Lemma}[section]
\newtheorem{remark}{Remark}[section]
\numberwithin{equation}{section}

\usepackage{lineno}

\usepackage{babel}

\journal{LINEAR ALGEBRA AND ITS APPLICATIONS}

\begin{document}

\begin{frontmatter}

\title{Asymptotics of eigenvalues of large symmetric Toeplitz matrices with smooth simple-loop symbols}

\author{A.A. Batalshchikov}
\address{ASM Solutions LLC, Nyzhnyaya Pervomayskaya Str. 48/9, Moscow, Russia, e-mail: abatalshikov@gmail.com}

\author{S.M. Grudsky}
\address{CINVESTAV, Departamento de Mathematicas, Ciudad de Mexico 07360, Mexico, e-mail: grudsky@math.cinvestav.mx,  tel.+52 55 5747 3800 ext.6457}

\author{I.S.~Malisheva}
\address{344090, South Fedederal University, Rostov-on-Don, Milchakova str. 8-A, Russia}

\author{S.S.~Mihalkovich}
\address{344090, South Fedederal University, Rostov-on-Don, Milchakova str. 8-A, Russia}

\author{E.~Ram\'{i}rez de Arellano}
\address{CINVESTAV, Departamento de Mathematicas, Av. Instituto Politécnico Nacional 2508, Mexico, tel.+52 55 5061 3869}

\author{V.A.~Stukopin}
\address{1) Moscow Institute of Physics and Technology (MIPT)}
\address{2) SMI of VSC RAS (South Mathematical Institute of Vladikavkaz Scientific Center of Russian Academy of Sciences)}
\address{3) Interdisciplinary  Scientific Center J.-V. Poncelet (CNRS UMI 2615) and  Scoltech    Center for Advanced Study}

\begin{abstract}
This paper is devoted to the asymptotic behavior of all eigenvalues of  Symmetric (in general non Hermitian) Toeplitz matrices with moderately smooth symbols which trace out a simple loop on the complex plane line as the dimension of the matrices increases to infinity. The main result describes the asymptotic structure of all eigenvalues. The constructed expansion is uniform with respect to the number of eigenvalues. \\
\end{abstract}


\end{frontmatter}


\section{Introduction}

Given a function $a(t)$ in $L^1$ on the complex unit circle $\mathbb{T}$ we denote by $a_l$ the l-th Fourier coefficient
\begin{equation}
a_l = \frac{1}{2\pi} \int_0^{2\pi} a(e^{ix})e^{-ilx}dx,
\end{equation}
$l \in \mathbb{Z}$, and by $T_n(a)$ the $n \times n$ Toeplitz matrix $(a_{j-k})_{j,k=1}^n$.

The object of our study is the behavior of the spectral characteristics (eigenvalues and singular numbers, eigenvectors, determinants, condition numbers, etc.) of Toeplitz matrices in the case when the dimension of the matrices tends to infinity.
It has been intensively studied for a century (see \cite{grenander1958toeplitz}, \cite{Uni}, \cite{BS06}, \cite{BG2005}, and literature cited there).
This problem is important for statistical mechanics and other applications (\cite{grenander1958toeplitz}, \cite{DIK13}, \cite{Diaconis03patternsin}, \cite{Kad1966}, \cite{MW1973}).
First of all, we mention the numerous versions of the Szegö theorem on the asymptotic distribution of eigenvalues and theorems of Abram-Parter type on the asymptotic distribution of singular numbers (\cite{Parter1986}, \cite{Avram1988}, \cite{ZT2007}, \cite{BGM08.1}).
There is a rich literature devoted to the asymptotics of the determinants of Toeplitz matrices. (see monographs \cite{Uni}, \cite{BS06}, papers \cite{WiOT}, \cite{DIKann}, \cite{DIK2012}, \cite{DIK13} and literature cited there). 
Much attention has been paid to the asymptotics of the largest and smallest eigenvalues (\cite{KaMuSe}, \cite{Pa100}, \cite{Wi88}).

We note that articles on the individual asymptotics of all eigenvalues have appeared quite recently. Thus, cases of real-valued symbols (self-adjoint Toeplitz operators) satisfying the so-called SL (Simple-Loop) condition were studied in the articles \cite{Grudsky2011}, \cite{DIK13}, \cite{BBGM15}.
In these articles, the cases of polynomial, infinitely smooth, having 4 continuous derivatives of symbols were successively considered.

Finally, a symbol which has continuous first derivative and satisfies certain additional conditions at the minimum and maximum points is considered in \cite{BGM2017}. In \cite{Barrera2017}, the case of a symbol that has a 4th order zero and thus does not satisfy the conditions of a simple loop is studied.

The asymptotic expansions of all eigenvalues are constructed in the case of essentially complex-valued symbols having singularities of the Fisher-Hartwig type in the articles \cite{DaiGearyKadanoff}, \cite{Kadanoff2009}, \cite{BBG2012}, \cite{BBGM11}.
We note that the complex-valued (non self-adjoint) case is more complicated than the real-valued one, because finding the location of the limit set of eigenvalues of Toeplitz matrices with $n$ tending to infinity is a nontrivial question.
This question is resolved in \cite{WiOT} for the case of the Fisher-Hartwig singularities
considered in the above-mentioned papers. 
In this paper it is shown that the limit set coincides with the image of the symbol in the complex plane.

There is another well-known case, when the limit set also coincides with the symbol image.
It is the case when this image is a ``curve without an interior''. 
Using this fact, in \cite{BGS2015} we solved the problem of the asymptotic behavior of the eigenvalues of Toeplitz symmetric matrices with a polynomial symbol satisfying the following condition. Namely, the symbol passes its own curve-image exactly two times, when the variable $t$ makes one turn on the unit circle.  

Note that the asymptotic structure of the eigenvectors in the case of $n \to \infty$ is considered in the papers \cite{BGRS2015}, \cite{BBGM2012}, \cite{BGM2011}.

In this paper we generalize the results of the article \cite{BGS2015}, extending the class of symbols from polynomials to the class of smooth functions that have only two continuous derivatives. 
For this purpose the method used in \cite{BGS2015} needed a significant change.
The main obstacle here is that the considered symbol is defined only on the unit circle and does not allow, in general, unlike the case of the polynomial symbol of \cite{BGS2015}, an analytic continuation to the neighborhood of the unit circle $\mathbb{T}$.
At the same time, the eigenvalues are not located on the image-curve of the symbol $a(t)$, but they are located in some of its neighborhoods.
Therefore, the question arises about the continuation of $a(t)$ to the complex plane.
In this regard, we replace the symbol with a polynomial approximation (first terms of the Laurent series) of $a_n(t)$ of degree $(n-1)$ (see (\ref{eq:2.4})) and note that the operator corresponding to the Toeplitz matrix does not change.
The function $a_n(t)$ is considered in an annulus with the width of order $1/n$, containing $\mathbb{T}$.
We show that all eigenvalues lie in the image of the mapping $w = a_n(z)$ of this annulus.
On the one hand, it is necessary to transfer the methods and results of \cite{BBGM15}, \cite{BGM2017} from the real segment to a region in the complex plane.
On the other hand, we ensure that all constructions are uniform with respect to the parameters of the family of functions $\{a_n\}$, $n \in \mathbb {N}$.

The paper is organized as follows. 
Section 2 contains the main results of the work.
We consider in Section 3 an example with numerical calculations of all eigenvalues for different values of $n$. The presented figures bring up several questions about the location the eigenvalues. The main results that are formulated in Section 2 allows to answer these questions. In particular we give the asymptotics of the eigenvalues that are located near the points $z=a(1)$ and $z=a(-1)$ 
(see Lemma~\ref{lm:lambda_asymp_1}), where the derivative of the symbol vanishes. 
This result is a generalization to the complex case of the well-known results about the asymptotics
of the smallest and largest eigenvalues of large Toeplitz matrices with real value symbols 
(see \cite{Parter1986}, \cite{Wi88}).

Section 4 presents the results on the smoothness properties of the functions $a(t)$ and $a_n(t)$ that we need and the functions $b(t, \lambda)$, $b_n(t, \lambda)$ are constructed on the basis of $a(t)$, $a_n(t)$ (see (\ref{eq:b_hat_def}), (\ref{eq:b_n_hat_def})). In Section 5, a nonlinear equation is introduced for determining the eigenvalues and then its asymptotic properties are investigated. Section 6 is devoted to the analysis of the solvability of the above mentioned nonlinear equation in the complex domain surrounding the image-curve of the symbol $a(t)$.
The main results are proved in section 7.

\section{The main results}

Let $\alpha \ge 0$. We denote by $W^{\alpha}$ the weighted Wiener algebra of all complex-valued functions $a: \mathbb{T} \rightarrow \mathbb{C}$, whose Fourier coefficients satisfy
\begin{equation}
\| a \|_{\alpha} := \sum_{j= -\infty}^{\infty}|a_j|(|j| + 1)^{\alpha} < \infty.
\end{equation}

Let $m$ be the integer part of $\alpha$. It is readily seen that if $a \in W^{\alpha}$ then the function $g$ defined by $g(\varphi): = a(e^{i\varphi})$ is a $2\pi$-periodic $C^m$ function on $\mathbb{R}$. In what follows we consider complex-valued symmetric simple-loop functions in $W^{\alpha}$. To be more precise, for $\alpha \geq 2$, we let $\mathsf{CSL}^{\alpha}$ be the set of all $a \in W^{\alpha}$ such that $g(\varphi) = a(e^{\mathrm{i}\varphi})$ has the following properties:\\
\newcounter{N1}
\begin{list}{(\arabic{N1})}{\usecounter{N1}}
\item \label{condition11}
the function $g(\varphi)$ is symmetrical in the following sense:
\begin{equation} \label{cond1eq}
g(\varphi) = g(2\pi - \varphi), \qquad \varphi \in [0, 2\pi] .
\end{equation}
(It is equivalent to the conditon $a_j=a_{-j}$, $j = 1, 2, \ldots$ .)

\item \label{condition12}
$\Im(g)$ is a simple (without self-intersections) arc with non-coincident end points $M_0, M_1$: $g(0) = g(2\pi) = M_0$, $g(\pi) = M_1$, $M_0 \ne M_1$, so that
$g'(\varphi) \ne 0$ for $\varphi \in (0,\pi)$ and $g''(0) = g''(2\pi) \ne 0$,  $g''(\pi) \ne 0$.
\end{list}

It should be noted that if we have (\ref{cond1eq}) then
\begin{equation}
    \label{gCond1conseq}
    g'(0) = g'(2\pi) = g'(\pi) = 0.
\end{equation}

We introduce the following notation. Let $f(t) = \sum_{j=-\infty}^{\infty}f_j t^j$ $(t \in \mathbb{T})$
be a function from the space $L_2(\mathbb{T})$ such that
$\sum_{j=-\infty}^{\infty} |f_j|^2 < \infty$.
We consider the projectors
\[
\left[P_n f \right](t) := \sum_{j=0}^{n-1}f_j t^j, \qquad n=1, 2, \ldots .
\]

We will also denote the image of the operator $P_n$ by $L_2^{(n)}$.
Note that for the symbol $a \in L_{\infty}(\mathbb{T})$ the Toeplitz matrix $T_n(a)$
can be identified with the operator
\[
    T_n(a):L_2^{(n)} \to L_2^{(n)}, \text{\ defined by \ } T_n(a)f = P_n(a f).
\]

We introduce then the functions
\begin{equation} \label{eq:2.4}
a_n(t)=\sum_{j=-(n-1)}^{n-1}a_j t^j
\end{equation}
and note that
\begin{equation}
    T_n(a) = T_n(a_n).
\end{equation}

Therefore, we will use the function $a_n(t)$ instead of the symbol $a(t)$, when it will be convenient, and respectively the function $g_n(\varphi) := a_n(e^{\mathrm{i} \varphi})$ instead of $g(\varphi)$.

Note that the functions $a_n(t)$ and $g_n(\varphi)$ satisfy all conditions of the definition of $\mathsf{CSL}^\alpha$ for a sufficiently large $n$.
Besides, if $a(t) \in W^\alpha$ then
\begin{equation}
    \label{anEstim1}
    \sup_{t\in \mathbb{T}} \left| a(t)-a_n(t) \right| = o\left( \sfrac{1}{n^\alpha} \right), \quad n \to \infty
\end{equation}
(see Lemma \ref{lm:f_w_alpha}, i) below).

Introduce the sets:
\begin{equation}
    \label{r_a_set}
    \mathcal{R}(a):= \{g(\varphi):\ \varphi \in (0, \pi)\};
\end{equation}
\begin{equation}
    \label{p_n_a_set}
    \Pi_n(a) = \left\{ \psi=\varphi+\mathrm{i}\delta\, \vert \, \varphi \in [c n^{-1}, \pi-c n^{-1}],\ \delta \in [-C n^{-1}, C n^{-1}] \right\},
\end{equation}
where $c$ small enough and $C$ large enough are fixed positive numbers. Let us denote
\begin{equation}
    \mathcal{R}_n(a) := \{g_n(\psi):\ \psi \in \Pi_n(a)\}.
\end{equation}

It is well known (see for example \cite{Uni}) that the limit spectrum of the operator family
$\left\{ T_n(a)\right\}_{n=1}^\infty$ coincides with the curve $\mathcal{R}(a)$. 
Thus, for sufficiently large $n$ the spectrum of $T_n(a)$
is located in the neighborhood of $\mathcal{R}(a)$. Moreover, we will show that
$\mathrm{sp} T_n(a) \subset \mathcal{R}_n(a)$.

According to the conditions (\ref{condition11})-(\ref{condition12}), for any $\lambda \in \mathcal{R}(a)$ there exists exactly one $\varphi_1(\lambda) \in (0,\pi)$
such that $g(\varphi_1(\lambda))=\lambda$.
The symmetry implies that the function $\varphi_2(\lambda):=(2\pi - \varphi_1(\lambda)) \in (\pi,2\pi)$ also has this property: $g(\varphi_2(\lambda))=\lambda$.

For all $\lambda_0 \in \mathcal{R}(a)$ consider the function
\begin{equation}
    \label{eq:b_hat_def}
    \hat{b}(t,\lambda_0) = \frac{(a(t)-\lambda) e^{\mathrm{i}\varphi_1(\lambda_0)}}
    {(t-e^{\mathrm{i}\varphi_1(\lambda_0)}) (t^{-1} - e^{\mathrm{i}\varphi_1(\lambda_0)})}.
\end{equation}
(Note that $e^{-\mathrm{i} \varphi_2(\lambda)} = e^{\mathrm{i} \varphi_1(\lambda)}$,
therefore the function $\left( t^{-1} - e^{\mathrm{i} \varphi_1(\lambda)}\right)$ goes to zero at a single point $t_2 = e^{\mathrm{i} \varphi_2(\lambda)}$).
In a similar manner, for all $\lambda \in \mathcal{R}_n(a)$ there exist
$\varphi_{1,n}(\lambda) \in \Pi_n $, such that
\begin{equation}
    \label{eq:gn29}
    g_n(\varphi_{1,n}(\lambda)) = g_n(2 \pi - \varphi_{1,n}(\lambda)) = \lambda.
\end{equation}
It will be shown below that the point $\varphi_{1,n}(\lambda)$, satisfying condition (\ref{eq:gn29})
(see Lemma \ref{lm:g_n_psi}) is unique.
    
Together with (\ref{eq:b_hat_def}) we consider the function
\begin{equation}
    \label{eq:b_n_hat_def}
    \hat{b}_n(t,\lambda) = \frac{(a_n(t)-\lambda) e^{\mathrm{i}\varphi_{1,n}(\lambda)}}
    {(t-e^{\mathrm{i}\varphi_{1,n}(\lambda)}) (t^{-1} - e^{\mathrm{i}\varphi_{1,n}(\lambda)})}, \qquad \lambda \in \mathcal{R}_n(a),
\end{equation}
which is a polynomial of powers of $t$ and $t^{-1}$ of finite degree
and does not vanish in the domain $t \in \mathbb{T}$, $\lambda \in \mathcal{R}_n(a)$. 
We show that this function allows a Wiener-Hopf factorization of the form
\[
\hat{b}_n(t, \lambda) = \hat{b}_{n,+}(t, \lambda) \hat{b}_{n,+}(t^{-1}, \lambda),
\]
where $\hat{b}_{n,+}(t, \lambda)$ is a polynomial of degree $(n-2)$ of the variable $t$.
We introduce the function
\begin{equation}
    \label{eq:theta_n_def}
    \theta_n(\lambda) = \log \frac{\hat{b}_{n,+}(e^{\mathrm{i}\varphi_{1,n}(\lambda)}, \lambda)}
    {\hat{b}_{n,+}(e^{-\mathrm{i}\varphi_{1,n}(\lambda)}, \lambda)}, \qquad \lambda \in \mathcal{R}_n(a).
\end{equation}

Note that in the case of $\lambda \in \mathcal{R}(a)$, ($e^{\mathrm{i} \varphi_1(\lambda)} \in \mathbb{T}$) (\ref{eq:theta_n_def}) the function $\theta_n(\lambda)$ can be represented as
\begin{equation}
    \label{eq:teta_n_lambda}
    \theta_n(\lambda) = \frac{1}{2 \pi \mathrm{i}} \int_{\mathbb{T}}
    \frac{\log \hat{b}_n(\tau, \lambda)}{\tau - e^{\mathrm{i}\varphi_{1,n}(\lambda)}}d\tau
    -\frac{1}{2 \pi \mathrm{i}} \int_{\mathbb{T}}
    \frac{\log \hat{b}_n(\tau, \lambda)}{\tau - e^{-\mathrm{i}\varphi_{1,n}(\lambda)}}d\tau .
\end{equation}
We also introduce the function
\begin{equation}
    \label{eq:teta_lambda}
    \theta(\lambda) = \frac{1}{2 \pi \mathrm{i}} \int_{\mathbb{T}}
    \frac{\log \hat{b}(\tau, \lambda)}{\tau - e^{\mathrm{i}\varphi_1(\lambda)}}d\tau
    -\frac{1}{2 \pi \mathrm{i}} \int_{\mathbb{T}}
    \frac{\log \hat{b}(\tau, \lambda)}{\tau - e^{-\mathrm{i}\varphi_1(\lambda)}}d\tau,
    \quad \lambda \in \mathcal{R}(a),
\end{equation}
where the integrals in (\ref{eq:teta_n_lambda}) and (\ref{eq:teta_lambda})
are understood in the sense of the principal value.
It is more convenient for us to consider the introduced functions as functions of the parameter
$s := \varphi_{1,n}(\lambda)$
($\lambda = g_n(s)$):
\begin{equation}
    \label{eq:eta_s_def}
    \eta(s) := \theta(g(s)), \qquad s \in (0,\pi)
\end{equation}
and
\begin{equation}
    \label{eq:eta_n_s}
    \eta_n(s) := \theta_n(g_n(s)), \qquad s \in \Pi_n(s)
\end{equation}

Introduce the values
\begin{align}
    \label{eq:d_j_n_def}
    d_{j,n} &= \frac{\pi j}{n+1}, \qquad j=1, 2, \ldots, n,
    \\
    \label{eq:e_j_n_def}
    e_{j,n} &= d_{j,n} - \frac{\eta_n(d_{j,n})}{n+1}, \qquad j=1, 2, \ldots, n.
\end{align}

We will also need the following small areas:
\begin{equation}
    \label{eq:pi_j_n_def0}
    \Pi_{j,n}(a) := \left\{  s \in \Pi_{n}(a),\ |s - e_{j,n}| \le \frac{c_n}{n+1} \right\},
    \quad j=1, 2, \ldots, n,
\end{equation}
where the constants $c_n$ does not depend on $j$ and decrease to 0 with $n \to \infty$ (see (\ref{eq:pi_j_n_def1})).
Now we are ready to formulate the main results of this work.
Let $\lambda_j^{(n)}$, $j=1, 2, \ldots, n$ be a numeration of the eigenvalues of the operator $T_n(a)$.
\begin{thm} \label{thm:a1}
Let $a$ be a symbol such that $a \in \mathsf{CSL}^\alpha$, $\alpha \ge 2$. Then, for a sufficiently large
natural number $n$, the following statements hold:
\begin{enumerate}[i)]
    \item all eigenvalues $T_n(a)$ are different, and $\lambda_{j}^{(n)}\in g(\Pi_{j,n})$ for $j=1, 2, \ldots, n$,
    \item the values $s_{j,n}$ such that $\lambda_j^{(n)} = g_n(s_{j,n})$ satisfy the equation
    \begin{equation}
        \label{eq:sjn1}
        (n+1) s + \eta_n(s) = \pi j + \Delta_n(s),
        \qquad j=1, 2, \ldots, n
    \end{equation}
    with $|\Delta_n(s)|=o(1/n^{\alpha-2})$ where $\Delta_n(s) \to 0$
    as $n \to \infty$ uniformly respect to $s \in \Pi_n(a)$.
    \item Equation (\ref{eq:sjn1}) has a unique solution in the domain $\Pi_{j,n}$.
\end{enumerate}
\end{thm}

\begin{thm}
\label{thm:sjn1}
Under the conditions of the Theorem \ref{thm:a1}
\[
    s_{j,n} = d_{j,n} + \sum_{k=1}^{[\alpha]-1}
    \frac{p_k(d_{j,n})}{(n+1)^k} + \Delta_2^{(n)}(j)
\]
where
\[
\Delta_2^{(n)}(j) =
\begin{cases}
o(1/n),\ \alpha=2 \\
O(1/n^{\alpha-1}), \ \alpha>2
\end{cases}
\]
as $n \to \infty$ uniformly in $j$. The functions  $p_k$ can be calculated explicitly;
in particular
\[
    p_1(s) = -\eta(s), \quad p_2(s)= \eta(s)\eta'(s).
\]
\end{thm}

\begin{thm}
\label{thm:lj1}
Under the conditions of Theorem \ref{thm:a1}
\begin{equation} \label{eq:lambda1}
    \lambda_j^{(n)} = g(d_{j,n}) + \sum_{k=1}^{[\alpha]-1}
    \frac{r_k(d_{j,n})}{(n+1)^k} + \Delta_3^{(n)}(j)
\end{equation}
where
\[
\Delta_3^{(n)}(j) =
\begin{cases}
o\left( \frac{d_{j,n}(\pi-d_{j,n})}{n} \right), \qquad \alpha = 2, \\
O\left( \frac{d_{j,n}(\pi-d_{j,n})}{n^{\alpha-1}} \right), \qquad \alpha > 2.
\end{cases}
\]
as $n \to \infty$ uniformly in $j$.
The coefficients $r_k$ can be calculated explicitly; in particular
\[
r_1(\varphi) = -g'(\varphi)\eta(\varphi) \quad \text{and} \quad
r_2(\varphi) = \frac12 g''(\varphi)\eta^2(\varphi) + g'(\varphi) \eta(\varphi) \eta'(\varphi).
\]

\end{thm}

\section{Numerical example and the consequences of the main results}

Define a symbol by the function $g_1(\varphi)$ ($=a_1(e^{\mathrm{i} \varphi})$):
\[
g_1(\varphi)=c_{1}\sin(c_{0}\varphi^{2}) + c_2 ((1+\varphi)^{5/2}+(1-\varphi)^{5/2}),\qquad \varphi\in[-\pi;\pi],
\]
(it is more convenient for us to consider the symbol in this section on the segment $[-\pi;\pi]$)
where:
\[
c_0=\frac{1}{5}-\frac{1}{6}\mathrm{i},\qquad  c_1=\frac{(1-\pi)^{3/2}-(\pi+1)^{3/2}}{16\pi c_{0}\cos(\pi^{2}c_{0})},
\qquad c_2 = \frac{1}{20}.
\]
The expression for the constants $c_{1,2}$ are derived from the conditions
$g'(-\pi) = g'(\pi)$.
(The equalities $g(-\pi)=g(\pi)$ and $g''(-\pi) = g''(\pi)$ are a consequence of the symmetry of function $g(\varphi)$.)
It can be verified that the constructed function satisfies the conditions (1) and (2) at the begining of Section 2.
(The condition $g'(\varphi) \ne 0$ for $\varphi \in (-\pi;\pi)$ can be verified numerically.)

We can see that the third derivative of the symbol $a_1(t)$ has singularities at the points
$t = e^{\mathrm{i}}$ and $t = e^{-\mathrm{i}}$.
It is easy to see that $a_1(t) \in \mathsf{CSL}^{2.5 - \delta}$
for arbitrary small $\delta > 0$.

The image of the symbol $a_1(t)$ and the eigenvalues of the matrices $T_n(a_1)$ for
$n=20$ and $n=80$ are shown in the Figure~\ref{fig:a1_n20}
and Figure~\ref{fig:a1_n80} correspondingly.
If we look at these Figures then we can make the following observations:
\begin{enumerate}
    \item The limit set of the eigenvalues for $T_n(a_1)$ if $n \to \infty$ really coincide with 
    $\mathcal{R}(a_1)$.

    \item The points of concentration for the eigenvalues of $T_n(a_1)$ are $z_1 = a_1(1)$
    and $z_2 = a_1(-1)$. The distance between consecutive eigenvalues in neighborhoods
    of $z_1$ and $z_2$ is much less than outside of these neighborhoods.
    
\item Some eigenvalues are located under the curve $\mathcal{R}(a_1)$ and others above the curve.
\end{enumerate}

\begin{figure}
\begin{centering}
\includegraphics[scale=0.6]{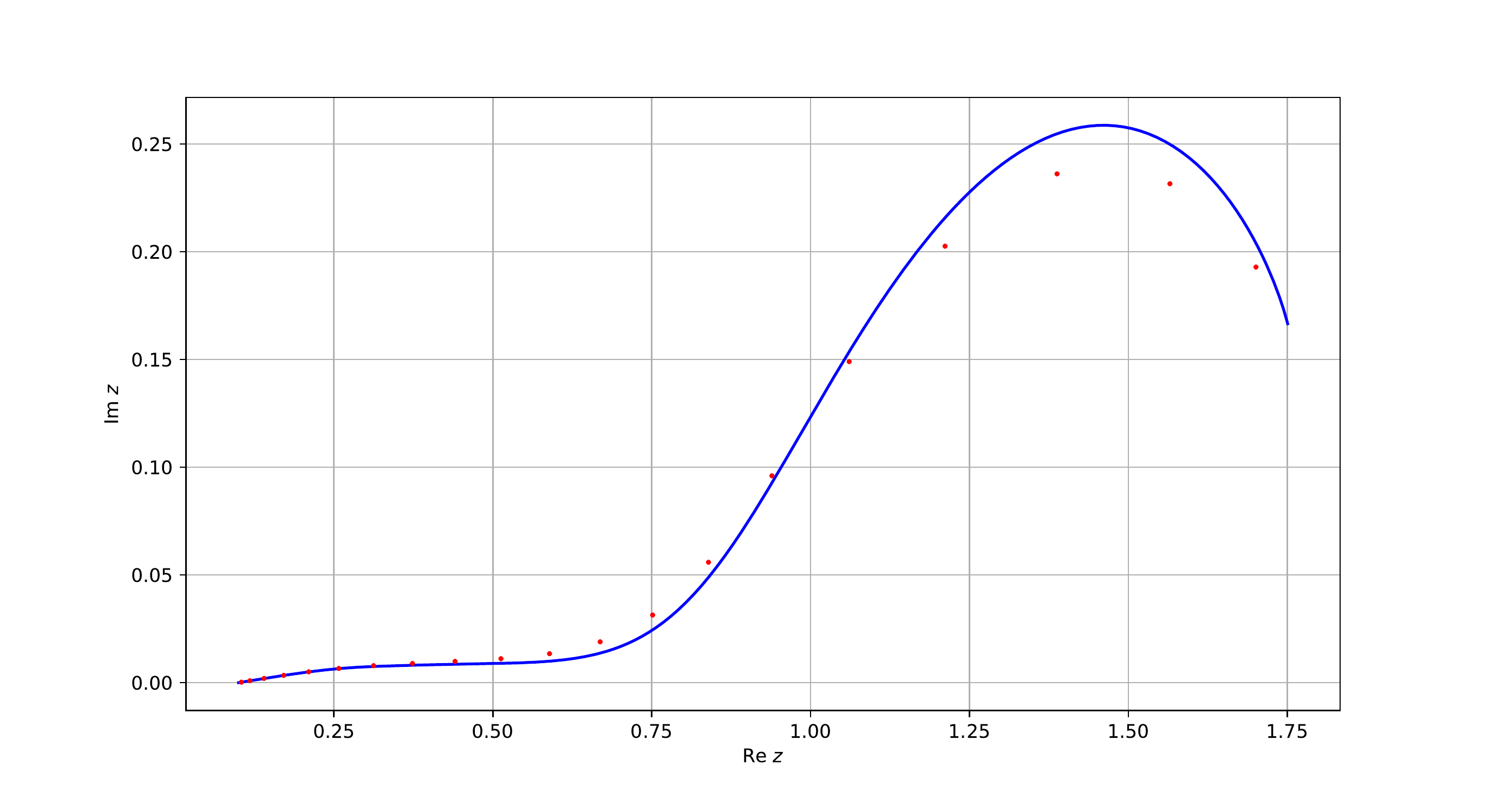}
\par\end{centering}
\caption{\label{fig:a1_n20}Image of the symbol $a_1(t)$ and eigenvalues 
of the matrix $T_{20}(a_1)$}
\end{figure}

\begin{figure}
\begin{centering}
\includegraphics[scale=0.6]{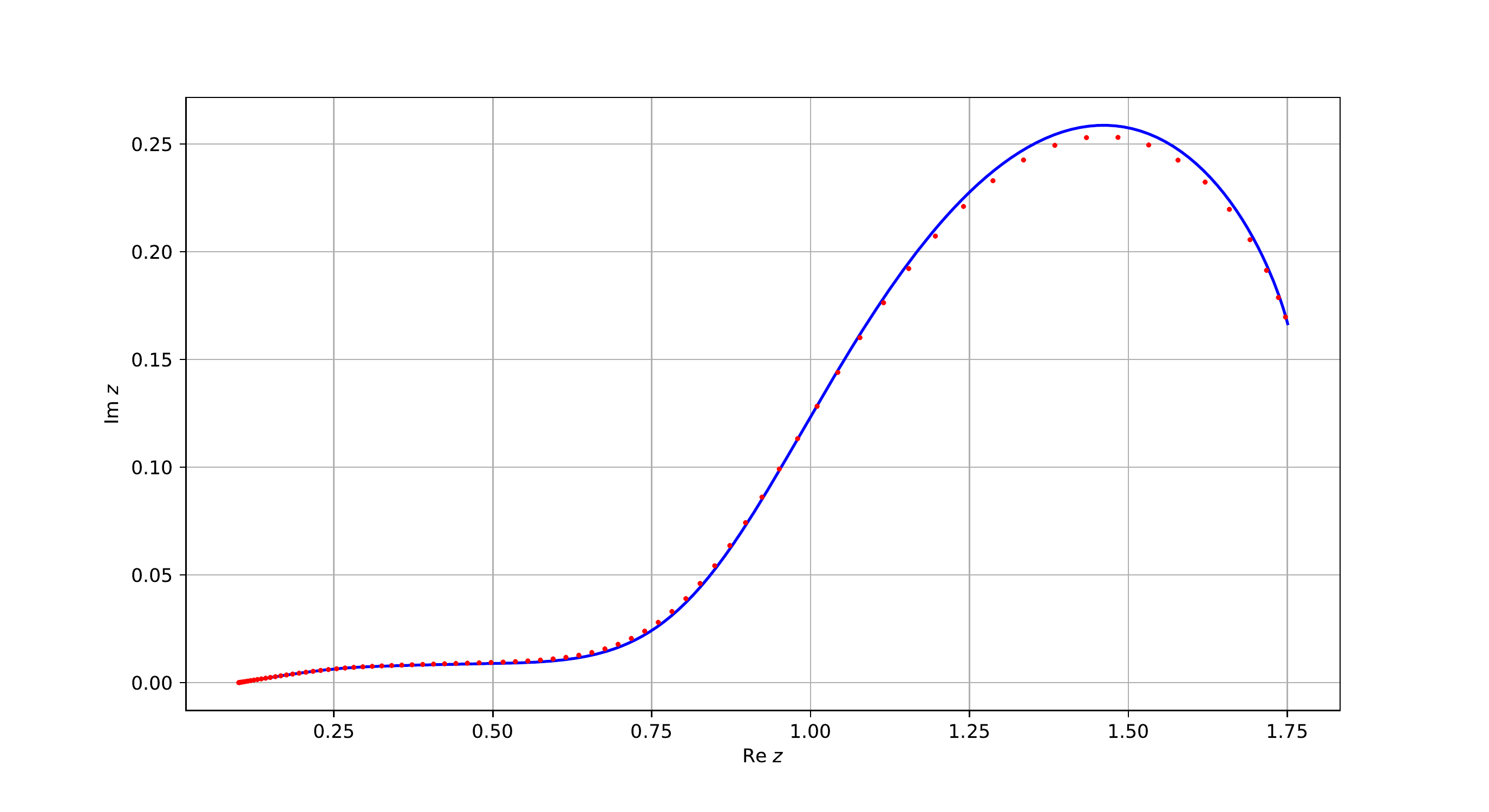}
\par\end{centering}
\caption{\label{fig:a1_n80}Image of the symbol $a_1(t)$ and eigenvalues 
of the matrix $T_{80}(a_1)$}
\end{figure}

We are going to show that Theorem~\ref{thm:lj1} allows to explain and clarify these observations.

Designate the spectrum of $T_n(a)$ by 
$\mathrm{Sp}(T_n(a)) := \left\{ \lambda_{j}^{(n)} \right\}_{j=1}^{n}$.
Then the limit set of the eigenvalues of the sequence $\left\{ T_n(a) \right\}_{n=1}^\infty$
is the following:
\[
\Lambda(a) := \limsup_{n \to \infty} \mathrm{Sp}(T_n(a)).
\]
The next Lemma is a direct consequence of the Theorem~\ref{thm:lj1}.

\begin{lm}  \label{lm:lim_set}
Under the conditions of Theorem~\ref{thm:a1}
\[
\Lambda(a) = \mathcal{R}(a).
\]
In addition
\[
\sup_{1\le j \le n} \rho\left( \lambda_j^{(n)}, \mathcal{R}(a) \right) \le \frac{\mathrm{const}}{n} .
\]
(Here $\rho(z, \mathcal{R}(a))$ is the distance between the point $z$ and the curve $\mathcal{R}(a)$ in the complex plane).
\end{lm}

Consider now the observation 2. The situation in the neighborhoods of the points $z_1$ and $z_2$
clarifies the following statement.

\begin{lm}  \label{lm:lambda_asymp_1}
Let the conditions of the Theorem~\ref{thm:a1} be fulfilled. Then:
\begin{enumerate}[i)]
    \item if $\frac{j}{n+1} \to 0$, then the following
    asymptotic formula is true:
    \[
    \lambda_j^{(n)} = g(0) + \frac{\pi^2 g''(0)}{2} \frac{j^2}{(n+1)^2} + \Delta_4^{(n)}(j),
    \]
    where
    \begin{equation} \label{eq:delta_4}
        \Delta_4^{(n)}(j) = o\left( \sfrac{j^2}{n^2} \right) .
    \end{equation}

    \item if $\frac{n+1-j}{n+1} \to 0$, then the following
    asymptotic formula is true:
    \[
    \lambda_j^{(n)} = g(\pi) + \frac{\pi^2 g''(\pi)}{2} \frac{(n+1-j)^2}{(n+1)^2} + \Delta_5^{(n)}(j),
    \]
    where
    \begin{equation} \label{eq:delta_5}
        \Delta_5^{(n)}(j) = o\left( \frac{(n+1-j)^2}{n^2} \right) .
    \end{equation}
\end{enumerate}
\end{lm}

\begin{proof}
Since $a(t) \in W^2$ then, according to (\ref{eq:lambda1}), we have:
\[
    \lambda_j^{(n)} = g(d_{j,n}) - \frac{g'(d_{j,n}) \eta(d_{j,n})}{n+1}
    + o\left( \frac{d_{j,n}(\pi - d_{j,n})}{n} \right) .
\]
Consider i). We have in this case that
\[
o\left( \frac{d_{j,n}(\pi - d_{j,n})}{n} \right) = o\left( \frac{j}{n^2} \right).
\]
Taking into account that
\[
g'(d_{j,n}) = g(0) + \frac{g''(0)}{2}(d_{j,n})^2 + o(d_{j,n})^2 ,
\]
$g'(d_{j,n}) = O(d_{j,n})$, and $\eta(d_{j,n}) = O(d_{j,n})$,
one has
\[
\frac{|g'(d_{j,n}) \eta(d_{j,n})|}{n+1} = o\left( \frac{j^2}{n^3} \right) .
\]
Thus we get i).
Case ii) is proved analogously.
\end{proof}

Thus Lemma~\ref{lm:lambda_asymp_1} shows that the distance between
consecutive eigenvalues located in a neighborhood of the point $z_1 = g(0)$
$(j / n  \to 0)$ is of $o(j/n^2)$.
So if $j$ is bounded (``the first eigenvalues'') then this distance is of
\begin{equation} \label{eq:lambda_first_dist}
o\left( \sfrac{1}{n^2} \right).
\end{equation}

The same result is true for neighborhoods of the point $z_2 = g(\pi)$
(``the last eigenvalues'').
On the other hand for $\varepsilon < d_{j,n} < \pi - \varepsilon$,
where $\varepsilon > 0$ is fixed small enough (inner eigenvalues)
Theorem \ref{thm:lj1} give that
\begin{equation} \label{eq:lambda_dist}
| \lambda_{j+1}^{(n)} - \lambda_{j}^{(n)} | = o\left( \sfrac{1}{n} \right).
\end{equation}

Thus (\ref{eq:lambda_first_dist}) and (\ref{eq:lambda_dist}) confirm
and explain of the observation 2 above.

Pass to observation 3. Consider ``inner eigenvalues'' that is 
$\varepsilon < d_{j,n} < \pi - \varepsilon$.
Suppose that $a(t) \in W^3$
the formula (\ref{eq:lambda1}) give to us that
\[
\lambda_j^{(n)} = g(d_{j,n}) - \frac{g'(d_{j,n}) \eta(d_{j,n})}{n + 1}
+ O\left( \sfrac{1}{n^2} \right) .
\]
Let $\tilde{e}_{j,n} := d_{j,n} - \dfrac{\Re \eta(d_{j,n})}{n+1}$, then
\begin{equation}
\label{eq:lambda_e_j_n}
\lambda_j^{(n)} = g(\tilde{e}_{j,n}) - \mathrm{i} \frac{g'(\tilde{e}_{j,n}) \Im \eta(d_{j,n})}{n + 1}
+ O\left( \sfrac{1}{n^2} \right) .
\end{equation}

Thus $\lambda_j^{(n)}$ is located on the normal to the curve $\mathcal{R}(a)$ at the point 
$z = g(\tilde{e}_{j,n})$ with exactitude $O\left( 1 / n^2 \right)$.

In addition we can see that the point $\lambda_j^{(n)}$ is located ``above'' or ``below''
of the curve $\mathcal{R}(a)$ depending the sign of the real number $\Im \eta(d_{j,n})$.
Moreover, formula (\ref{eq:lambda_e_j_n}) give us that
$\dfrac{|g'(\tilde{e}_{j,n}) \Im \eta(d_{j,n})|}{n + 1}
+ O\left( 1/ n^2 \right)$ is the distance between $\lambda_j^{(n)}$ and $\mathcal{R}(a)$.

It should be noted that Theorem~\ref{thm:lj1} has not only a qualitative sense but also a qualitative (numerical) one.
If one has numerical values for the function $\eta(d_{j,n})$ we can calculate all eigenvalues 
$\lambda_{j}^{(n)}$, $j=1, 2, \ldots, n$ very rapidly for different values of $n$.
This idea was applied in the article \cite{Ekstrom_2017} (in case of real valued symbols)
where the function $\eta(d_{j,n})$ was calculated with the help of the eigenvalues 
$\lambda_{j}^{(n)}$ for not very large $n$.

In the rest of this section we illustrate the accuracy of asymptotic formulas.
Introduce the notation for approximated eigenvalues from (\ref{eq:lambda1}):
\begin{align*}
    \lambda_{1,j}^{(n)} &= g(d_{j,n}) - \frac{g'(\varphi)\eta(\varphi)}{n+1}, \\
    \lambda_{2,j}^{(n)} &= g(d_{j,n}) - \frac{g'(\varphi)\eta(\varphi)}{n+1}
    +\frac{\frac12 g''(\varphi)\eta^2(\varphi) + g'(\varphi) \eta(\varphi) \eta'(\varphi)}{(n+1)^2}.
\end{align*}

The relative approximation error will be characterized by the following values:
\begin{align*}
\Delta_1^{(n)} &= \max_{j=\overline{1,n}} \left| \frac{\lambda_{1,j}^{(n)}-\lambda_j^{(n)}}{\lambda_j^{(n)}} \right|, \\
\Delta_2^{(n)} &= \max_{j=\overline{1,n}} \left| \frac{\lambda_{2,j}^{(n)}-\lambda_j^{(n)}}{\lambda_j^{(n)}} \right|.
\end{align*}

The resulting accuracy of the spectrum of $T_n(a_1)$ is shown in table \ref{tab:delta_a1}.

\begin{table}[!ht]
  \centering
\begin{tabular}{|c|c|c|c|c|c|}
\hline
$n$ & 20 & 40 & 80 & 160 & 320\tabularnewline
\hline
$\Delta_1^{(n)}$ & 3.2e-03 & 8.8e-04 & 2.3e-04 & 5.9e-05 & 1.5e-05 \tabularnewline
\hline
$\Delta_2^{(n)}$ &  3.9e-04 & 5.6e-05 & 7.2e-06 & 9.2e-07 & 1.2e-07 \tabularnewline
\hline
\end{tabular}
  \caption{Comparative accuracy of the calculation of the spectrum $T_n$ for the symbol $a_1(t)$}
  \label{tab:delta_a1}
\end{table}

We can see for even $n=20$ the accuracy good enough for both approximations.

\section{Preliminary results}

Let the function $f$ be $f(t) = \sum_{j=-\infty}^{\infty} f_j t^j \in L_2(\mathbb{T})$.
We introduce the operator $Q_n$ by the formula
\begin{equation}
    \label{eq:qnf}
    \left( Q_n f \right) := \sum_{j=n}^\infty f_j t^j .
\end{equation}

The following Lemma holds.

\begin{lm} \label{lm:f_w_alpha}
\begin{enumerate}[i)]
    \item If $f \in W^\alpha(\mathbb{T})$, $\alpha \ge 0$, then
    \[
    \sup_{t \in \mathbb{T}} \left| \left( Q_n f \right)(t)\right| \le
    \frac{\|Q_n f \|_\alpha}{n^\alpha}.
    \]
    \item For a natural number  $k \le \alpha$, then
    \[
    f^{(k)}(t) \in W^{\alpha-k}(\mathbb{T})
    \quad \text{and} \quad
    \sup_{t \in \mathbb{T}} \left|\left(Q_n f^{(k)}\right)(t)\right| \le
    \frac{\|Q_n f\|_\alpha}{n^{\alpha-k}} .
    \]
    \item For a real number  $k > \alpha$, then
    \[
    \sum_{j=-(n-1)}^{n-1} |f_j| |j|^k \le 2 n^{k-\alpha} \|f\|_\alpha.
    \]
\end{enumerate}
\end{lm}

\begin{proof}
Items  i) and ii) were proved in \cite{BBGM15}, \cite{BGM2017}. Let us prove iii):
\begin{align*}
\sum_{j=-(n-1)}^{n-1} |f_j| |j|^k & = \sum_{j=-(n-1)}^{n-1}
\left( |f_j| (1 + |j|)^\alpha \right) \frac{|j|^k}{(1 + |j|)^\alpha}\\
& \le (1+n)^{k-\alpha} \sum_{j=-n}^{n} |f_j| (1 + |j|)^\alpha
\le 2 n^{k-\alpha} \|f\|_\alpha .
\end{align*}
\end{proof}

Let $a \in W^\alpha$. Consider the functions
\[
g(\varphi) = a(e^{\mathrm{i}\varphi}) = \sum_{j=-\infty}^{\infty} a_j e^{\mathrm{i}j\varphi}
\quad \text{and} \quad
g_n(\varphi) = a_n(e^{\mathrm{i}j\varphi}) = \sum_{j=-(n-1)}^{(n-1)} a_j e^{\mathrm{i}j\varphi}.
\]

The following Lemma gives an asymptotic representation when $n \to \infty $ of the function $g_n(\varphi)$  in the complex domain $\Pi_n(a)$ using the function $g(\varphi)$ (and its derivatives) defined on $[0,\pi]$.

\begin{lm} \label{lm:norm_alpha_n_k}
Let the point $\psi = \varphi + \mathrm{i} \delta \in \Pi_n(a)$.
If  $a(t) \in W^\alpha$, $\alpha \ge 0$,  $m = [\alpha]$ then
\[
g_n(\psi) = g(\varphi) + \sum_{k=1} ^{m} g^{(k)}(\varphi) (\mathrm{i} \delta)^k
+ \sum_{k=0} ^{m+1} \alpha_{n,k}(\varphi) (\mathrm{i} \delta)^k,
\]
where $\alpha_{n,k}(\varphi) \in W^0$,
\[
\|\alpha_{n,k}\|_0 = o(n^{k-\alpha}), \qquad k=0, 1, \ldots, m
\]
and
\[
\|\alpha_{n,m+1}\|_0 = O(n^{m+1-\alpha}).
\]
\end{lm}

\begin{proof}
We can represent $g_n(\psi)$ in the form:
\[
g_n(\psi) = \sum_{j=-{n-1}}^{n-1} a_j e^{\mathrm{i}j(\varphi + \mathrm{i}\delta)}.
\]

As $|j \delta| \le C$, we can use the following asymptotics
\[
e^{\mathrm{i}j(\mathrm{i}\delta)} = 1 + \sum_{k=1}^{m} \frac{(-\delta j)^k}{k!}
+ O (\delta j) ^ {m+1}.
\]
So we obtain
\begin{align*}
g_n(\psi) &= \sum_{j=-(n-1)}^{n-1} a_j e^{\mathrm{i} j \varphi}
\left( 1 + \sum_{k=1}^{m} \frac{(-\delta j)^k}{k!} + O(\delta j)^{m+1} \right)\\
&= g(\varphi) + \alpha_{n,0}(\varphi) + \sum_{k=1}^{m} (\mathrm{i}\delta)^k
\sum_{j=-(n+1)}^{(n-1)} a_j e^{\mathrm{i} j \varphi} (i j)^k +
\sum_{j=-(n+1)}^{n+1} a_j\, O(\delta j)^{m+1} \\
&= g(\varphi) + \sum_{k=1}^{m} g^{(k)}(\varphi) (\mathrm{i} \delta)^k
+ \alpha_{n,0}(\varphi) + \sum_{k=1}^{m} (\mathrm{i} \delta)^k \alpha_{n,k}(\varphi)
+ \delta^{m+1} \alpha_{n,m+1}(\varphi) .
\end{align*}

It's easy to see, due to Lemma \ref{lm:f_w_alpha}, that
\[
\|\alpha_{n,0}(\varphi)\|_0 =  \sum_{|j| > (n-1)} |a_j|
\le (1+n)^{-\alpha} \sum_{|j| > (n-1)} |a_j| (1 + |j|)^\alpha = o(n^{-\alpha}).
\]
Similarly, for $k=1, 2, \ldots, m$
\[
\|\alpha_{n,k}(\varphi)\|_0 =  \sum_{|j| > (n-1)} |a_j| |j|^k
\le (1+n)^{k-\alpha} \sum_{|j| > (n-1)} |a_j| (1 + |j|)^\alpha = o(n^{k-\alpha}).
\]
Finally,
\begin{align*}
\|\alpha_{n,m+1}(\varphi)\|_0 & \le \mathrm{const}\, \sum_{j=-(n-1)}^{n-1}
|a_j| (1 + |j|)^{m+1} \\
&\le \mathrm{const}\, n^{m+1-\alpha} \sum_{j=-(n-1)}^{n-1}
|a_j| (1 + |j|)^\alpha = O(n^{m+1-\alpha}) .
\end{align*}
\end{proof}

Now we can prove the correctness of the introduction of the value $\varphi_{1,n}(\lambda)$ (see (\ref{eq:gn29}))

\begin{lm}
\label{lm:g_n_psi}
Let  $a(t) \in W^{\alpha}$, $\alpha \ge 2$. Then mapping
\[
g_n(\psi) :\ \Pi_n(a)\ \to \ \mathcal{R}_n(a)
\]
is a bijection for a sufficiently large natural number $n$.
\end{lm}

\begin{proof}
The surjectivity of this mapping follows from the definition of the set  $\mathcal{R}_n(a)$. \\
We will prove injectivity by contradiction. Suppose that for each $n$ there exists couple of different points $\psi_{1,n}, \; \psi_{2,n} \in \Pi_n(a)$ and $\lambda_n \in \mathcal{R}_n(a)$ such that
\begin{equation}
    \label{eq:g_n_psi_2}
    g_n(\psi_{1,n}) = g_n(\psi_{2,n}) = \lambda_n.
\end{equation}
Without loss of generality, we can assume that the sequences
$\left\{ \psi_{1,n} \right\}_{n=1}^\infty$, $\left\{ \psi_{2,n} \right\}_{n=1}^\infty$
and  $\left\{ \lambda_{n} \right\}_{n=1}^\infty$
have limits  $\varphi_1,\, \varphi_2 \in [0,\pi]$ respectively, and
$\lambda_0 \in \mathcal{R}(a)$. It is obvious that  $\varphi_1 = \varphi_2 := \varphi_0$
because of $g(\varphi_1) = g(\varphi_1) = \lambda_0$, and mapping
$g(\varphi):\ [0,\pi] \to \mathcal{R}(a)$  is injective by condition  (\ref{condition12}).
Suppose that  $\varphi_0 \notin \{0,\pi\}$. Then by  (\ref{eq:g_n_psi_2}) we have
\begin{equation}
    \label{eq:g_diff_0}
    g_n(\psi_{2,n}) - g_n(\psi_{1,n}) = 0.
\end{equation}
Applying Taylor's formula, we obtain 
\[
g_n'(\psi_{1,n})(\psi_{2,n} - \psi_{1,n}) + o\left(|\psi_{2,n} - \psi_{1,n}|\right) = 0.
\]
The latter is impossible since
\[
\lim_{n\to\infty} g_n'(\psi_{1,n}) = g'(\psi_0) \ne 0.
\]
Thus, $\varphi_0 = 0$ or $\varphi_0 = \pi$.
Let's suppose for definiteness, that $\varphi_0=0$. Then for a sufficiently large natural number  $n$
$|\psi_{1,n}| \le \sigma$ and $|\psi_{2,n}| \le \sigma$ where $\sigma > 0$ is a small number. We will show that in this case the equality (\ref{eq:g_diff_0}) is also impossible.
Indeed
\begin{equation}
    \label{eq:g_n_diff_2}
    g_n(\psi_{2,n}) - g_n(\psi_{1,n}) = \int_{I_n} g_n'(s) ds
    = \int_{I_n} \left( \int_{[0,s]} g_n''(u) du \right) ds,
\end{equation}
where $I_n := [\psi_{1,n}, \psi_{2,n}]$ and $[0,s]$ are segments of the complex plane connecting these points.
We note that, identity  (\ref{eq:g_n_diff_2}) is true because $g_n'(0) = 0$.
We rewrite (\ref{eq:g_n_diff_2}) as
\begin{equation}
    \label{eq:g_n_diff_3}
    g_n(\psi_{2,n}) - g_n(\psi_{1,n}) = \frac{g_n''(\psi_{1,n})}{2}(\psi_{2,n}^2 - \psi_{1,n}^2)
    + \int_{I_n} \left( \int_{[0,s]} (g_n''(u) - g_n''(\psi_{1,n})) du \right) ds .
\end{equation}
Let us estimate the integral term. Since $a(t) \in W^\alpha(\mathbb{T})$, $\alpha \ge 2$,
then  $a''(t) \in W^{\alpha-2}(\mathbb{T})$.

Therefore, the function  $g''_n(\psi)$ is uniformly continuous in $\Pi_n(a)$ with respect to $n$.
Thus, the following estimate holds for the difference in the integral expression (\ref{eq:g_n_diff_3})
\[
\sup_{|u|\le\sigma,\, |\psi_{1,n}|\le \sigma}
|g''_n(u) - g''_n(\psi_{1,n})| = o(1), \qquad \sigma \to 0
\]
 which is uniform with respect to $n$. Replacing the variables in (\ref{eq:g_n_diff_3}) by
$u = e^{\mathrm{i}\arg s} v$
and
$s = p \psi_{2,n} + (1-p) \psi_{1,n}$
we obtain
\begin{align*}
& \left| \int_{I_n} \left( \int_{[0,s]} (g_n''(u) - g_n''(\psi_{1,n})) du \right) ds \right| =
\\ & \qquad =|\psi_{2,n} - \psi_{1,n}| \left| \int_0^1 e^{\mathrm{i} \arg s} \left( \int_0^{|s|}
\left( g_n''(e^{\mathrm{i} \arg s} v) - g_n''(\psi_{1,n}) \right) dv \right) dp \right|
\\  & \qquad\le |\psi_{2,n} - \psi_{1,n}| o(1) \int_0^1 \left( \int_0^{|s|} dv \right) dp
\le o(|\psi_{2,n} - \psi_{1,n}|) \int_0^1 |p\psi_{2,n} + (1-p) \psi_{1,n}| dp
\\ & \qquad \le o(|\psi_{2,n} - \psi_{1,n}|) \left( |\psi_{2,n}| + |\psi_{1,n}|\right).
\end{align*}

Thus from (\ref{eq:g_n_diff_3}) we get that for a sufficiently large $n$:
\[
|g_n(\psi_{2,n}) - g_n(\psi_{2,n})| \ge
|\psi_{2,n} - \psi_{1,n}| \left( \frac{|g''(0)|}{4} |\psi_{2,n} + \psi_{1,n}|
- o(|\psi_{2,n}| + |\psi_{1,n}|)\right) .
\]
Further, taking in account that
\[
\psi_{1,n} = \varphi_{1,n} + \mathrm{i} \delta_{1,n},
\qquad
\psi_{2,n} = \varphi_{2,n} + \mathrm{i} \delta_{2,n},
\]
where $\varphi_{1,2,n} \ge \frac{c}{n}$ and $|\delta_{1,2,n}| \le \frac{C}{n}$, we obtain
\begin{align*}
\dfrac{|g_n(\psi_{2,n}) - g_n(\psi_{2,n})|}{|\psi_{2,n} - \psi_{1,n}|} & \ge
\frac{|g''(0)|}{4} \sqrt{(\varphi_{1,n} + \varphi_{2,n})^2 + (\delta_{1,n} + \delta_{2,n})^2}
-o\left( \sqrt{\varphi_{1,n}^2 + \delta_{1,n}^2} + \sqrt{\varphi_{2,n}^2 + \delta_{2,n}^2}
\right)
\\
& > \frac{|g''(0)|}{4} \left( \varphi_{2,n} + \varphi_{1,n} \right) -o\left( \varphi_{1,n} +
|\delta_{1,n}| + \varphi_{2,n} + |\delta_{2,n}|\right)
\\
& \ge \frac{|g''(0)|}{4}\cdot \frac{2c}{n} - o\left( \sfrac{1}{n}\right) > 0.
\end{align*}
Consequently, $| g_n(\psi_{2,n}) - g_n(\psi_{1,n})| > 0$, which contradicts (\ref{eq:g_diff_0}).
Case  $\varphi_0 = \pi$ is treated similarly.
\end{proof}

\begin{remark}
\label{rem:34}
The function $g_n(\psi)$ is analytic, therefore it is a conformal mapping of the domain $\Pi_n(a)$ onto $\mathcal{R}_n(a)$.
\end{remark}

Now consider the function $\hat{b}(t, \lambda_0)$ of the form   (\ref{eq:b_hat_def}), where $\lambda_0 \in \mathcal{R}(a)$.
It is convenient to consider it in two forms. As a function, the second argument of which is
$s_0=\varphi_1(\lambda_0)\ (=g^{-1}(\lambda_0))\ \in (0,\pi)$:
\begin{equation}
    \label{eq:b_t_s0}
    b(t,s_0) := \hat{b}(t, g(s_0)) = \frac{\left(a(t) - g(s_0) \right) e^{\mathrm{i} s_0}}
    {(t - e^{\mathrm{i} s_0})(t^{-1}- e^{\mathrm{i} s_0})}
\end{equation}
and as a function, the second argument of which is $\tau = e^{\mathrm{i} s_0} \in \mathbb{T}$:
\begin{equation}
    \label{eq:b_tilde_t_tau}
    \tilde{b}(t,\tau) := \frac{\left(a(t) - a(\tau) \right) \tau}
    {(t - \tau)(t^{-1} - \tau)}.
\end{equation}

Similarly, we will consider (\ref{eq:b_n_hat_def}) in the form
\begin{equation}
    \label{eq:b_n_t_s}
    b_n(t,s) := \frac{\left(a_n(t) - g_n(s) \right) e^{\mathrm{i} s}}
    {(t - e^{\mathrm{i} s})(t^{-1}- e^{\mathrm{i} s})}, \qquad s \in \Pi_n(a)
\end{equation}
and in the form
\begin{equation}
    \label{eq:b_n_tilde_t_tau}
    \tilde{b}_n(t,\tau) := \frac{\left(a_n(t) - a_n(\tau) \right) \tau}
    {(t - \tau)(t^{-1} - \tau)}, \qquad \tau \in \mathcal{R}_n(a) .
\end{equation}

The following Lemma gives conditions for the functions introduced above and for their
partial derivatives to be in $W^\alpha$.
\begin{lm}
\label{lm:b_in_w}
Let  $a \in \mathsf{CSL}^{\alpha}$, $\alpha \ge 2$.  If  $s_0 \in (0, \pi)$, and $s \in \Pi_n(a)$, then
\begin{enumerate}[i)]
\item $\tilde{b}(\cdot, \tau) \in W^{\alpha-2}$, $\tilde{b}(t, \cdot) \in W^{\alpha-2}$ and
\begin{equation}
    \label{eq:b_tilde_norm_1}
    \| \tilde{b}(\cdot, \tau) \|_{\alpha-2} \le \mathrm{const}\, \| a \|_\alpha, \quad
    \| \tilde{b}(t, \cdot) \|_{\alpha-2} \le \mathrm{const}\, \| a \|_\alpha,
\end{equation}
\begin{equation}
    \label{eq:b_n_norm_1}
    \| b_n(\cdot, s)\|_{\alpha-2} \le \mathrm{const}\, \| a \|_\alpha
\end{equation}
and
\begin{equation}
    \label{eq:b_diff_norm_1}
    \| b(\cdot, s_0) - b_n(\cdot, s_0) \|_{\alpha-2} \le \mathrm{const}\, n^{-(\alpha-2)};
\end{equation}

\item for $\alpha \ge 2 + p + \ell$, we have that
\[
\frac{\partial^{p+\ell} \tilde{b}(\cdot, \tau) }{\partial^p t \partial^\ell \tau} \in W^{\alpha-2-p-\ell},
\qquad \frac{\partial^{p+\ell} \tilde{b}(t, \cdot) }{\partial^p t \partial^\ell \tau} \in W^{\alpha-2-p-\ell},
\]
and
\begin{equation}
    \label{eq:3_14}
    \left\| \frac{\partial^{p+\ell} \tilde{b}(\cdot, \tau) }{\partial^p t \partial^\ell \tau} \right\|_{\alpha-2-p-\ell}
    \le \mathrm{const}\, \| a \|_\alpha; \qquad
    \left\| \frac{\partial^{p+\ell} \tilde{b}(t, \cdot) }{\partial^p t \partial^\ell \tau} \right\|_{\alpha-2-p-\ell}
    \le \mathrm{const}\, \| a \|_\alpha,
\end{equation}
\begin{equation}
    \label{eq:3_15}
    \left\| \frac{\partial^{p+\ell} b_n(\cdot,s) }{\partial^p t \partial^\ell s} \right\|_{\alpha-2-p-\ell}
    \le \mathrm{const}\, \| a \|_\alpha,
\end{equation}
\begin{equation}
    \label{eq:3_16}
    \left\| \frac{\partial^{p+\ell} b(\cdot,s_0) }{\partial^p t \partial^\ell s_0}
    - \frac{\partial^{p+\ell} b_n(\cdot,s_0) }{\partial^p t \partial^\ell s_0}
    \right\|_{\alpha}
    \le \mathrm{const}\, n^{-(\alpha-2-p-\ell)};
\end{equation}

\item for $\alpha < 2 + p + \ell$, we have
\begin{equation}
    \label{eq:3_17}
    \left\| \frac{\partial^{p+\ell} b_n(\cdot,s_0) }{\partial^p t \partial^\ell s_0} \right\|_{0}
    \le \mathrm{const}\, n^{2+p+\ell-\alpha} \| a \|_\alpha.
\end{equation}
\end{enumerate}

Here all values of ``const'' do not depend on $\tau $, $t$, $s$, $s_0$ and $n$, respectively.
\end{lm}

\begin{proof}
We represent the function  $b(t,s)$ in the form
\[
b(t,s_0) = c_0(s_0) t \left( b_1(t,s_0) - b_2(t,s_0) \right),
\]
where
\[
b_1(t,s_0) := \frac{a(t) - g(s_0)}{t-e^{\mathrm{i}s_0}}, \quad
b_2(t,s_0) := \frac{a(t) - g(s_0)}{t-e^{-\mathrm{i}s_0}}, \quad
c_0(s_0) = \frac{1}{2\mathrm{i} \sin s_0}.
\]
Since $g(s) = g(-s)$, we get
\begin{align*}
b(t,s_0) &= c_0(s_0) t \sum_{j= -\infty, j \neq 0}^{\infty} a_j \left( \frac{t^j - e^{\mathrm{i} j s_0} }{t - e^{\mathrm{i} s_0}}
- \frac{t^j - e^{-\mathrm{i} j s_0} }{t - e^{-\mathrm{i} s_0}} \right) := b^{(+)}(t, s_0) + b^{(-)}(t, s_0),
\end{align*}
Here the $ b^{(\pm)}(t,s_0)$ are responsible for summation over $j>0$ and $j<0$, respectively.
We estimate the first term.
\begin{align*}
& b^{(+)}(t, s_0) = c_0(s_0) t \sum_{j=1}^{\infty} a_j \sum_{k=0}^{j-1} t^k \left(e^{\mathrm{i} (j-1-k) s_0}
- e^{-\mathrm{i} (j-1-k) s_0} \right)\\
&= 2 \mathrm{i} c_0(s_0) t \sum_{j=1}^{\infty} a_j \sum_{k=0}^{j-2} t^k \sin (j-1-k) s_0 .
\end{align*}

Changing the order of summation, we have
\begin{equation}
    \label{eq:b_3_18}
    b^{(+)}(t,s_0) = t \sum_{k=0}^{\infty} \left( \sum_{j=k+2} ^{\infty} a_j
    \frac{\sin (j-1-k) s_0}{\sin s_0} \right) t^k .
\end{equation}
Let us get the estimate of type (\ref{eq:b_tilde_norm_1}).
For this we use the following inequality:
\begin{equation}
    \label{eq:sin_l_s_estim}
    \left| \frac{\sin (L\cdot s)}{\sin s} \right| \le \mathrm{const}\,  L,
    \qquad L>0,
\end{equation}
which is true for all real $s$ (in particular for $s= s_0 \in [0, \pi]$), and complex $s$,
such that
\begin{equation}
    \label{eq:l_im_s_estim}
    |L\cdot \Im s| \le M,
\end{equation}
where $M>0$ is a fixed number. Then
\begin{align*}
\| b^{(+)}(\cdot, s_0)\|_{\alpha-2} &\le \mathrm{const}\, \sum_{k=0}^{\infty}
\left( \sum_{j=k+2}^{\infty} |a_j|\cdot |j-1-k| \right) (2+k)^{\alpha-2}
\\ &= \mathrm{const}\, \sum_{j=2}^{\infty} |a_j|
\left( \sum_{k=0}^{j-2} |j-1-k| \cdot (2+k)^{\alpha-2} \right).
\end{align*}
It is not difficult to show that
\[
\sum_{k=0}^{j-2} |j-1-k| (2+k)^{\alpha-2} \le \mathrm{const}\, (1+j)^\alpha,
\]
then
\begin{equation}
    \label{eq:b_norm_3_21}
    \| b^{(+)}(\cdot, s_0)\|_{\alpha-2} \le \mathrm{const}\, \sum_{j=2}^{\infty} |a_j| (1+j)^\alpha
    = \mathrm{const}\, \|a\|_\alpha .
\end{equation}

Similarly, it can be shown that
\begin{equation*}
\|b^{(-)}(\cdot, s_0) \|_{\alpha - 2} \leq \mathrm{const}\, \cdot \sum_{-\infty}^{j=-2} |a_j| (1 + |j|)^{\alpha}.
\end{equation*}

Obviously, the last two inequalities imply the fulfillment of the first inequality (\ref{eq:b_tilde_norm_1}).
The second relation (\ref{eq:b_tilde_norm_1}) is true by symmetry
$\tilde{b}(t, \tau) = \tilde{b}(\tau, t)$.
The inequality (\ref{eq:b_diff_norm_1}) also follows from (\ref{eq:b_tilde_norm_1}) if instead of $a(t)$ we take the difference $(a(t) - a_n(t))$ and use the statement of the Lemma \ref{lm:f_w_alpha}, i).
The inequality (\ref{eq:b_n_norm_1}) is proved in the same way as in (\ref{eq:b_norm_3_21}), it should only be replaced $s_0$ with $s$ in (\ref{eq:b_3_18}),
infinite upper limits by $j$ to $(n-1)$, and take into account that for $s \in \Pi_n(a)$ the condition (\ref{eq:l_im_s_estim}) is satisfied because
$|\Im s| \le \dfrac{C}{n}$, and value  $(j-1-k) < n$.

Let us prove item  ii). According to Lemma \ref{lm:f_w_alpha} ii), if $f(t) \in W^\alpha$, then
\[
\frac{\partial^p f}{\partial^p t}(t) \in W^{\alpha-p}
\quad \text{and}
\]
\begin{equation}
    \label{eq:norm_d_p_f}
    \left\| \frac{\partial^p f}{\partial^p t} \right\|_{\alpha-p} \le \mathrm{const}\, \|f \|_p .
\end{equation}

Therefore, in case $\ell = 0$, the first of the relations (\ref{eq:3_14}) is proved. By the symmetry of $\tilde{b}(t,\tau)$ we can argue that the second of the inequalities is proved.
(\ref{eq:3_14}) in case  $\ell=0$.
Now let $p = 0$, in case of the first of the relations (\ref{eq:3_14}), using (\ref{eq:b_3_18})
we have
\begin{equation}
    \label{eq:d_b_3_22}
    \frac{\partial^\ell}{\partial^\ell s_0} b^{(+)}(t,s_0) = \sum_{k=0}^{\infty}
    \left( \sum_{j=k+2}^{\infty} a_j \frac{\partial^\ell}{\partial^\ell s_0}
    \left( \frac{\sin(j-1-k)s_0}{\sin s_0} \right) t^k
    \right) .
\end{equation}

Then, by analogy with (\ref{eq:sin_l_s_estim}), we can use the inequality
\begin{equation}
    \label{eq:d_l_sin_3_23}
    \frac{\partial^\ell}{\partial s^\ell} \left( \frac{\sin(L\cdot s)}{s} \right) \le
        \mathrm{const}\, L^{\ell+1}
\end{equation}
which is true for all real $s=s_0$ and complex $s$ because (\ref{eq:l_im_s_estim}).
Then we get
\begin{align*}
\left\| \frac{\partial^\ell}{\partial s^\ell} b^{(+)}(t,s) \right\|_{\alpha-2-\ell}
&\le \mathrm{const}\, \sum_{k=0}^{\infty} \sum_{j=k+2}^{\infty} |a_j|\cdot |j-1-k|^{\ell+1}
(2+k)^{\alpha-2-\ell}
\\ &=\mathrm{const}\, \sum_{j=2}^{\infty} |a_j| \left( \sum_{k=0}^{j-2} |j-1-k|^{\ell+1}
(2+k)^{\alpha-2-\ell} \right)
\\ &=\mathrm{const}\, \sum_{j=2}^{\infty} |a_j| (1+j)^\alpha \le \mathrm{const}\, \|a\|_{\alpha}.
\end{align*}

The assertion for $b^{(-)}(t, s)$ and the other assertions ii) are proved by analogy with i) and by taking into account the property (\ref{eq:norm_d_p_f}).
We turn to the proof of iii). Differentiating (\ref{eq:d_b_3_22}) $p$ times over $t$, we obtain
\[
\frac{\partial^{p+\ell} b_n(t,s_0)}{\partial^p t \partial^\ell s_0}
= \sum_{k=0}^{n-1} \left( \sum_{j=k+2}^{n-1} a_j \frac{\partial^\ell}{\partial^\ell s_0}
\left( \frac{\sin(j-1-k)s_0}{\sin s_0} \right) \right)
\left( \prod_{v=0}^{p-1} (k-v) \right) t^{k-p} .
\]
Thus, using (\ref{eq:d_l_sin_3_23}), (replacing $\infty$ with $(n + 1)$), we get
\begin{align*}
\left\| \frac{\partial^{p+\ell} b_n(t,s_0)}{\partial^p t \partial^\ell s_0} \right\|_0
&\le \mathrm{const}\, \sum_{k=p}^{n-1} \left( \sum_{j=k+1}^{n-1} |a_j| \cdot |j-1-k|^{\ell+1} \right) k^p
\\ &\le \mathrm{const}\, \sum_{j=2}^{n-1} |a_j| \sum_{k=p}^{j-2} (j-1-k)^{\ell+1} k^p
\le \mathrm{const}\, \sum_{j=2}^{n-1} |a_j| (j+1)^{p+\ell+2}
\\ &\le \mathrm{const}\, n^{p+\ell+2-\alpha} \sum_{j=2}^{n-1} |a_j| (j+1)^\alpha
\le \mathrm{const}\, n^{p+\ell+2-\alpha} \|a\|_\alpha .
\end{align*}
The case of the function $b^{(-)}(t)$ is treated similarly.
\end{proof}

\begin{lm}
\label{lm:d_b_n}
Let $s = s_0 + \mathrm{i} \delta \in \Pi_n(a)$.
If $a(t) \in W^\alpha$, $p+\ell \le \alpha - 2$, $m=[\alpha-2-p-\ell]$ then
\[
\frac{\partial^{p+\ell}}{\partial^p t \partial^\ell s} b_n(t,s) =
\frac{\partial^{p+\ell}}{\partial^p t \partial^\ell s} b(t,s_0) +
\sum_{k=1}^{m} \frac{\partial^{p+\ell+k}}{\partial^p t \partial^{\ell+k} s} b(t,s_0) (\mathrm{i} \delta)^k
+ \sum_{k=0}^{m+1} \beta_{n,k}^{p,l}(t,s_0) (\mathrm{i} \delta)^k ,
\]
where $\beta_{n,k}^{p,l}(\cdot,s_0) \in W^0$ and besides
\begin{align*}
& \| \beta_{n,k}^{p,l}(\cdot,s_0) \|_0 = o\left(n^{k-\alpha-2-p-\ell} \right),
\quad k=0, 1, \ldots, m.
\\
& \| \beta_{n,m+1}^{p,l}(\cdot, s_0) \|_0 = O\left(n^{(m+1)-\alpha-2-p-\ell} \right)
\end{align*}
\end{lm}

The proof of this Lemma is carried out on the basis of the previous Lemma, similarly to the Lemma \ref{lm:norm_alpha_n_k}.

An important role in the theory of Toeplitz operators is played by the concept of the topological index of a function.
\begin{definition}
\label{def:wind}
Let the function $a(t)$ be continuous on the unit circle $\mathbb{T}$, and $a(t) \ne 0$, $t \in \mathbb{T}$,
then the topological index of the function $a(t)$ with respect to the point $z = 0$ is an integer of the form
\[
\mathrm{wind}\, a(t) := \left. \frac{1}{2\pi} \arg a(t) \right|_{\mathbb{T}}
\]
where $\left. \arg a(t) \right |_{\mathbb{T}}$ is the increment of the continuous branch of the argument of the function $a(t)$,
when the point $t$ makes a full turn on the curve $\mathbb{T}$ in the positive direction.
\end{definition}

Consider the problem about topological index of the functions $b(t,s)$ and $b_n(t,s)$.
\begin{lm}
\label{lm:wind_b_n}
Let  $a(t) \in W^\alpha$, $\alpha \geq 2$. Then for any  $s_0 \in (0,\pi)$  we have
\[
\mathrm{wind}\, b(t,s_0) = 0
\]
and for any  $s \in \Pi_n(a)$, we have
\[
\mathrm{wind}\, b_n(t,s) = 0 .
\]
\end{lm}
\begin{proof}
For any  $s_0 \in (0,\pi)$ the function  $b(t,s_0) \ne 0$.
In addition, due to the symmetry $b(e^{\mathrm {i} \varphi}, s_0) = b (e^{\mathrm{i}(2 \pi- \varphi)}, s_0) $ we can see that the image of $b(e^{\mathrm{i} \varphi}, s_0) $, $\varphi \in [0, 2\pi]$ represents a curve without interior, described twice: once and back,
when $\varphi$ describe the segments $[0, \pi]$ and $[\pi, 2\pi]$, respectively. Thus, the first relation in the formulation of Lemma \ref{lm:wind_b_n} is proved.
The second is proved similarly.
\end{proof}

Let us now consider  the sequence of functions  $\eta_n(s) $ of the type (\ref{eq:theta_n_def})- (\ref{eq:teta_n_lambda}), (\ref{eq:eta_n_s}) and compare it with the limit function (\ref{eq:teta_lambda})-(\ref{eq:eta_s_def})
\begin{equation}
    \label{eq:eta_s0_def}
    \eta(s_0) = \frac{1}{2 \pi \mathrm{i}} \int_{\mathbb{T}}
    \frac{\log b(\tau, s_0)}{\tau - e^{\mathrm{i}s_0}}d\tau
    -\frac{1}{2 \pi \mathrm{i}} \int_{\mathbb{T}}
    \frac{\log b(\tau, -s_0)}{\tau - e^{-\mathrm{i}s_0}}d\tau,
    \quad s_0 \in (0,\pi) .
\end{equation}

It's obvious that $b(\tau,s_0) \ne 0$ for $\tau \in \mathbb{T}$ and $s_0 \in [0,\pi]$.
Since  $\mathbb{T} \times [0,\pi]$ is a compact set we have
\begin{equation}
    \label{eq:inf_b_n}
    \inf_{s_0 \in [0,\pi]} \inf_{t \in \mathbb{T}} |b(t, s_0)| = \Delta > 0 .
\end{equation}
Thus, according to Lemma \ref{lm:d_b_n}, there is such a large enough natural number $N_0$, that
\begin{equation}
    \label{eq:inf_3}
    \inf_{n \ge N_0} \inf_{s \in \Pi_n(a)} \inf_{t \in \mathbb{T}} |b_n(t, s)| \ge \frac{\Delta}{2}.
\end{equation}

To analyze the functions $\eta_n(s)$, $\eta(s)$, we need generalized Hölder classes.
We say that $f(t) \in H^\mu(K)$, $0 < \mu \le 1$ where $K$ is a compact domain of the complex
plane $\mathbb{C}$, if the following condition is satisfied:
\[
\| f \|_{H^\mu} := \sup_{t \in K} |f(t)| + \sup_{t_1,\,t_2 \in K}
\frac{|f(t_2) - f(t_1)|}{|t_2 - t_1|^\mu} < \infty .
\]
We define the class  $\mathcal{C}^{m + \mu}(K)$. Let us say that $f(z) \in \mathcal{C}^{m + \mu}(K)$, $m=0, 1, 2, \ldots$,
$0 < \mu \le 1$ if $f^{(m)}(z) \in H^\mu(K)$.
Moreover, the norm of the function $f(t)$ in this space is introduced by the formula:
\[
\| f \|_{\mathcal{C}^{m + \mu}(K)} = \sum_{k=0}^{m-1} \sup_{t \in K} |f^{(k)}(t)| + \|f^{(m)}(t) \|_{H^\mu} .
\]
Note that  $\mathcal{C}^{0 + \mu}(K) = H^\mu(K)$. We also agree that $C^0 := W^0$.
In the following, we use the following known result. (see \cite{BBGM15},  Lemma 3.6).
\begin{lm}
\label{lm:f_in_c}
Let $f(t) \in W^\alpha$, $\alpha > 0$, then $f(t) \in \mathcal{C}^{m + \mu}(\mathbb{T})$, where $m = [\alpha]$, $\mu = \alpha - [\alpha]$.
\end{lm}

Introduce the following notation:
\[
\Lambda(t,s_0) := \frac{1}{2 \pi \mathrm{i}} \int_{\mathbb{T}}
    \frac{\log b(\tau, s_0)}{\tau - t}d\tau,
    \qquad s_0 \in (0,\pi)
\]
and
\[
\Lambda_n(t, s) := \frac{1}{2 \pi \mathrm{i}} \int_{\mathbb{T}}
    \frac{\log b_n(\tau, s)}{\tau - t}d\tau,
    \qquad s \in \Pi_n(a).
\]
Then we get that
\[
\eta(s_0) = \Lambda(e^{\mathrm{i} s_0}, s_0) - \Lambda(e^{-\mathrm{i} s_0}, s_0),
\qquad s_0 \in (0,\pi),
\]
and
\[
\eta_n(s) = \Lambda_n(e^{\mathrm{i} s}, s) - \Lambda_n(e^{-\mathrm{i} s}, s),
\qquad s \in \Pi_n(a).
\]

\begin{lm}
\label{lm:eta_props3}
Let the function $a(t) \in \mathsf{CSL}^{\alpha}$, $\alpha \geq 2$, then
\begin{enumerate}[i)]
\item $\eta(s_0) \in \mathcal{C}^{m+\mu}([0,\pi])$, where $m=[\alpha - 2]$, $\mu=\alpha-2-m$;
\item the point $s = s_0 + \mathrm{i} \delta \in \Pi_n(a)$ allows following representation 
\[
\eta_n(s) = \eta(s_0) + \sum_{k=1}^{m} \eta^{(k)}(s_0)(\mathrm{i} \delta)^{k}
+ \sum_{k=1}^{m+1} \gamma_{n,k}(s_0) (\mathrm{i} \delta)^{k}
\]
where
\[
|\gamma_{n,k}(s_0)| = o(n^{k-(\alpha - 2)}), \qquad k = 0, 1, \ldots, m,
\]
\[
|\gamma_{n,m+1}(s_0)| = O(n^{m+1-(\alpha - 2)})
\]
\end{enumerate}
\end{lm}
\begin{proof}
According to Lemma \ref{lm:wind_b_n}, for any $s_0 \in [0, \pi]$ we can choose a continuous branch of the  $\log b(\tau, s_0)$,
moreover this function will be continuous in $s_0 \in [0, \pi]$.
Further, according to the well-known theorem of the theory of Banach algebras
$\log b(\cdot, s_0) \in W^{\alpha-2}$ since $b(\cdot, s_0) \in W^{\alpha-2}$ and the relation (\ref{eq:inf_b_n}) is satisfied, and the norm  $\| b(\cdot, s_0) \|_{\alpha-2}$ is continuous in  $s_0$.
Note that the function $\Lambda(t,s_0)$ also has these properties, since the Cauchy singular integral operator involved in the definition of the function $\Lambda(t,s)$ is bounded in the space $W^{\alpha-2}$.

This implies that the function $\Lambda(t, s_0)$ has partial derivatives with respect to $t$
(with a fixed $s_0$), up to the order of $m$ and
$\dfrac{\partial^m}{\partial^m t} \Lambda(\cdot,s_0) \in H^\mu (\mathbb{T})$.
On the other hand, $\Lambda(t,s_0)$ has continuous partial derivatives with respect to $s_0$. Indeed:
\[
\frac{\partial^{p+\ell}}{\partial^p t \partial^\ell s_0} \Lambda(t,s_0) =
\frac{(p-1)!}{2 \pi \mathrm{i}} \int_{\mathbb{T}}
\frac{ \frac{\partial^\ell}{\partial^\ell s_0}\log b(\tau, s_0)}{(\tau - t)^p}d\tau .
\]
Applying integration by parts $p$ times to the last integral, we obtain 
\[
\frac{\partial^{p+\ell}}{\partial^p t \partial^\ell s_0} \Lambda(t,s_0) =
\frac{1}{2 \pi \mathrm{i}} \int_{\mathbb{T}}
\frac{ \frac{\partial^{p+\ell}}{\partial^p \tau \partial^\ell s_0}\log b(\tau, s_0)}
{\tau - t}d\tau .
\]
Thus, we obtain that if $p+\ell \le m$ ($m=[\alpha - 2]$), then
\[
\frac{\partial^{p+\ell}}{\partial^p t \partial^\ell s_0} \Lambda(\cdot,s_0) \in \mathcal{C}^{\alpha-2-p-\ell}(\mathbb{T})
\quad \text{and} \quad
\frac{\partial^{p+\ell}}{\partial^p t \partial^\ell s_0} \Lambda(t, \cdot) \in \mathcal{C}^{\alpha-2-p-\ell}([0,\pi]) .
\]
From the last two relations it follows that for $k=0, 1, \ldots, m$ we obtain:
\[
\frac{\partial^k \eta(s_0)}{\partial^k s_0} \in \mathcal{C}^{\alpha-2-k}([0,\pi]) .
\]
Indeed, taking for example, $k=2$, we get
\begin{align*}
\frac{\partial^2 \eta(s_0)}{\partial^2 s_0} &= \left( \left. \frac{\partial^2}{\partial^2 t}
 \Lambda(t,s_0) \right\vert_{t=e^{\mathrm{i} s_0}}
- \left. \frac{\partial^2}{\partial^2 t}
 \Lambda(t^{-1}, s_0) \right\vert_{t=e^{-\mathrm{i} s_0}}
\right)
\\ &+ 2 \left( \left. \frac{\partial^2}{\partial t \partial s_0}
 \Lambda(t,s_0) \right\vert_{t=e^{\mathrm{i} s_0}}
- \left. \frac{\partial^2}{\partial t \partial s_0}
 \Lambda(t^{-1},s_0) \right\vert_{t=e^{-\mathrm{i} s_0}}
\right)
\\ &+ \left( \left. \frac{\partial^2}{\partial^2 s_0}
 \Lambda(t,s_0) \right\vert_{t=e^{\mathrm{i} s_0}}
- \left. \frac{\partial^2}{\partial^2 s_0}
 \Lambda(t^{-1},s_0) \right\vert_{t=e^{-\mathrm{i} s_0}}
\right) .
\end{align*}
Similarly
\[
\frac{\partial^k \eta_n(s_0)}{\partial^k s_0} \in C^{\alpha-2-k}([0,\pi]),
\]
moreover
\begin{equation}
    \label{eq:d_k_eta_norm}
    \left\| \frac{\partial^k \eta(s_0)}{\partial^k s_0} - \frac{\partial^k \eta_n(s_0)}
    {\partial^k s_0}
    \right\|_{H^{\alpha-2-k}([0,\pi])}
    \le \mathrm{const}\, n^{-(\alpha-2-k)} ,
\end{equation}
where the ``const'' is independent of $n$.

Let now  $s = s_0 + \mathrm{i} \delta$, then
\begin{equation}
    \label{eq:eta_n_s_series}
    \eta_n(s) = \eta_n(s_0) + \sum_{k=1}^{m} \frac{\eta_n^{(k)}(s)}{k!} (\mathrm{i} \delta)^k
    + o_{m+1}(s_0, \delta) ,
\end{equation}
where $m=[\alpha - 2]$ and
\begin{equation}
    \label{eq:p_m_1_s_abs}
    |o_{m+1}(s_0, \delta)| \le \mathrm{const}\, \left| \eta_n^{(m+1)}(s_0) \right|\cdot |\mathrm{i} \delta|^{m+1} .
\end{equation}
According to (\ref{eq:d_k_eta_norm}) we have
\[
\eta_n^{(k)}(s_0) = \eta^{(k)}(s_0) + \gamma_{n,k}(s_0),
\]
where $|\gamma_{n,k}(s_0)| = o(n^{k-\alpha+2})$.
To estimate $o_{m+1}(\delta)$, we use an inequality of the form:
\[
\left\| \frac{d^{k+\ell} \left( \ln b(\cdot, s_0) \right)}{d^k t d^\ell s_0}  \right\|_0
\le \mathrm{const}\,  n^{(m+1)- \alpha + 2 - p - \ell } ,
\]
where  $m+1 = k + \ell > \alpha-2$  and ``const'' does not depend on $n$. Then it is easy to understand that
\[
|\eta^{(m+1)}_n(s_0) | \le \mathrm{const}\,n^{(m+1) - (\alpha-2)} .
\]
Thus, the Lemma  \ref{lm:eta_props3} is proved.
\end{proof}

\section{Equation for the eigenvalues}

In this section we derive an equation for finding the eigenvalues of the considered
Toeplitz matrices and subject this equation to asymptotic analysis when the parameter
$n \to \infty$. 
So, we consider the standard equation for finding eigenvalues and eigenvectors:
\begin{equation}
    \label{eq:t_n_x_n}
    T_n(a_n-\lambda)X_n=0, \qquad \lambda \in \mathcal{R}_n(a)
\end{equation}
in the space $L_2^{(n)}$.
Let us present the expression $a_n(t)-\lambda$ as a product $p(t,\lambda) \hat{b}_n(\cdot,\lambda)$,
where $\hat{b}_n(\cdot,\lambda)$
is the continuous non-degenerate zero index function and $p(t,\lambda)$ 
is a Laurent polynomial with three terms,
which inherits the zeros of the original function $a_n(t)-\lambda$.
Further, through some transformations we reduce (\ref{eq:t_n_x_n}) to an equation with an invertible operator
$T_{n+2}(\hat{b}_n(\cdot,\lambda))$ on the left-hand side.
By applying the operator $T^{-1}_{n+2}(\hat{b}_n(\cdot,\lambda))$ to this equation 
and considering the zeros of $p(t,\lambda)$ 
we get a homogeneous system of linear equations, the main determinant of which gives us the above-mentioned equation for finding eigenvalues.

We first prove the following result.
\begin{lm}
\label{lm:t_n_2_inv}
Let $a\in \mathsf{CSL}^\alpha$, $\alpha \ge 2$, then there is such a natural $N_0$,
independent of $\lambda \in \mathcal{R}_n(a)$ so that for all $\lambda \in \mathcal{R}_n(a)$ 
and for all $n \geq N_0$ the operator
$T_{n+2}(\hat{b}_n(\cdot,\lambda))$ is invertible and besides
\begin{equation}
    \label{eq:norm_t_n_m1}
    \left\| T^{-1}_{n+2} (\hat{b}_n(\cdot,\lambda)) \right\|_{L_2} \le M,
\end{equation}
where $M$ do not depend on $n$ and $\lambda \in \mathcal{R}_n(a)$.
\end{lm}
\begin{proof}
According to Lemma \ref{lm:wind_b_n}, $\mathrm{wind}\, \hat{b}(\cdot,\lambda_0) = 0$ for all $\lambda_0 \in \mathcal{R}_n(a)$.
Thus the finite section method (see, for  example \cite{Uni}) ensures the existence of a natural number $N_1(\lambda_0)$ such that for $n \ge N_1(\lambda_0)$ the operator
$T_{n}(\hat{b}(\cdot,\lambda_0))$ is invertible.
Since the set $\mathcal{R}(a)$ is compact, then, for all $\lambda_0$ from $\mathcal{R}(a)$, 
we can choose a single number
$N_1 = \sup_{\lambda_0 \in \mathcal{R}(a)} N_1(\lambda) < \infty$ such that for $n > N_1$
\begin{equation}
    \label{eq:norm_t_n_2_m1}
    \left\| T^{-1}_{n+2} (b(\cdot,\lambda_0)) \right\|_{L_2} \le \frac{M}{2},
\end{equation}
where $M$ does not depend on $n$ and $\lambda_0$.

According to the Lemma \ref{lm:d_b_n} for all small enough $\varepsilon > 0$ there is a number
$N_2$ such that for $n > N_2$
\begin{equation}
    \label{eq:norm_b_diff_0}
    \| \hat{b}(\cdot,\lambda_0) - \hat{b}_n(\cdot,\lambda) \|_0 < \varepsilon ,
\end{equation}
where $\lambda = g(s)$, $\lambda_0 = g(s_0)$, $s=s_0 + \mathrm{i} \delta$ ($\in \Pi_n(a)$),
and the numbers $\varepsilon$ and $N_2$ does not depend on $s_0$ and $\delta$.
Obviously (\ref{eq:norm_t_n_2_m1}) and (\ref{eq:norm_b_diff_0}) imply (\ref{eq:norm_t_n_m1}), where $N_0 = \max (N_1, N_2)$.
\end{proof}

We state the main result of this section.
Let $\lambda = g_n(s)$, $s \in \Pi_n(a)$. Then the following theorem holds.
\begin{thm}
\label{thm:4_2}
Let $a \in \mathsf{CSL}^{\alpha}$, $\alpha \ge 2$, and $n > N_0$ where $N_0$
is a rather large natural number.
Then $\lambda \in \mathcal{R}_n(a)$ is an eigenvalue of $T_n(a)$
if and only if
\begin{equation}
    \label{eq:4_4a}
    e^{-\mathrm{i}(n+1)s} \Theta_{n+2}\left( e^{\mathrm{i}s}, \lambda \right)
    \hat{\Theta}_{n+2}\left( e^{-\mathrm{i}s}, \lambda \right)
    - e^{\mathrm{i}(n+1)s} \Theta_{n+2}\left( e^{-\mathrm{i}s}, \lambda \right)
    \hat{\Theta}_{n+2}\left( e^{\mathrm{i}s}, \lambda \right) = 0 ,
\end{equation}
where the functions $\Theta_{n+2}$ and $\hat{\Theta}_{n+2}$ are defined by the formulas:
\[
\Theta_{n+2}(t, \lambda) = T^{-1}_{n+2}\left( \hat{b}_n(\cdot,\lambda)\chi_0\right)(t)
; \qquad
\hat{\Theta}_{n+2}(t, \lambda) = T^{-1}_{n+2}\left( \tilde{b}_n(\cdot,\lambda)\chi_0\right)(t^{-1}) ,
\]
and $\tilde{b}_n(t,\lambda) = \hat{b}_n(1/t, \lambda)$.
\end{thm}
\begin{proof}
Consider the equation
\[
T_n(a-\lambda) X_n = 0, \quad X_n \in L_2^{(n)}.
\]
Rewrite this equation in the following form
\begin{equation}
\label{eq:main_t_n}
T_n(a_n - \lambda)X_n = 0, \qquad \lambda \in \mathcal{R}_n(a).
\end{equation}
By the (\ref{eq:b_n_hat_def}) the above equation can be rewritten as follows:
\begin{equation}
\label{eq:4_6}
P_n\hat{b}_n(\cdot, \lambda) p(\cdot, \lambda)X_n = 0,
\end{equation}
where
\[
p(t, \lambda) = \left(t - e^{\mathrm{i} s}\right) \left(t^{-1} - e^{\mathrm{i} s}\right).
\]
(Recall that $g_n(s)=\lambda$.)
Multiply equality (\ref{eq:main_t_n}) by the base vector $\chi_1=t$. 
We obtain
\begin{equation}
\label{eq:4_7}
(P_{n+1} - P_1 )\hat{b}(\cdot, \lambda)p(\cdot, \lambda)\chi_1 X_n =0.
\end{equation}

Note that $P_{n+1} - P_1$ is a finite dimensional orthogonal space projector from $L_2(\mathbb{T})$ to the linear shell of vectors $\chi_1, \chi_2, \ldots, \chi_n$, where $\chi_j=t^j$.
It is easy to see that $p(\cdot, \lambda) \chi_1 X_n \in L_2^{(n+2)}$.
We set by definition
\[
Y_{n+2} := P_{n+2}(a_n - \lambda)\chi_1 X_n
= P_{n+2}\hat{b}_n(\cdot, \lambda)(\chi_1 p(\cdot, \lambda)X_n)
= P_{n+2}(\hat{b}_n(\cdot, \lambda))(\chi_1 p(\cdot, \lambda)X_n).
\]
Then equation (\ref{eq:4_7}) can be rewritten in the following form
\[
(P_{n+1} - P_1)Y_{n+2} = 0.
\]
The last equality means that the values of $Y$ have the following representation
\[ Y_{n+2} = y_0\chi_0 + y_{n+1}\chi_{n+1}. \]

According to the Lemma \ref{lm:t_n_2_inv}
the operator $T_{n+2}(\hat{b}_n(\cdot, \lambda))$ is invertible, so we get that
$$T_{n+2}^{-1}(\hat{b}_n(\cdot, \lambda))Y_{n+2} = \chi_1 p(\cdot, \lambda)X_n ,$$
that is
\begin{equation} \label{eq:4_7a}
y_0[T_{n+2}^{-1}(\hat{b}_n(\cdot, \lambda))\chi_0](t)
+ y_{n+1}[T_{n+2}^{-1}(\hat{b}_n(\cdot, \lambda))\chi_{n+1}](t) = t p(t, \lambda)X_n(t).
\end{equation}

We introduce the reflection operator acting on the space $L_2^{(n)}$:
\[
(W_n f_n)(t) = \sum_{j=0}^{n-1} f_{n-1-j} t^j,
\]
where
\[
f_n(t) = \sum_{j=0}^{n-1} f_{j} t^j.
\]
From the identity
 $W_{n+2} T_{n+2}(\hat{b}_n) W_{n+2} = T_{n+2}(\tilde{b}_n)$,
which is easy to verify by simple calculation, we get
\[
    [T_{n+2}^{-1}(\hat{b}_n(\cdot, \lambda))\chi_{n+1}](t) = t^{n+1} T^{-1}_{n+2}(\tilde{b}_n(\cdot, \lambda))(t^{-1}).
\]

Thus, equation (\ref{eq:4_7a}) we can rewrite in the form:
\begin{equation} \label{eq:4_8}
y_0\Theta_{n+2}(t, \lambda) + y_{n+1}t^{n+1}\hat{\Theta}_{n+2}(t, \lambda) = t p(t, \lambda)X_n(t).
\end{equation}
Considering that the multiplier $p(t, \lambda)$ disappears when $t = e^{\mathrm{i}s}$ 
and when
$t = e^{-\mathrm{i}s}$, we conclude that $y_0$ and $y_{n+1}$ must satisfy the following homogeneous system of linear algebraic equations:
\begin{align}
\label{eq:sys1}
&\Theta_{n+2}(e^{\mathrm{i}s}, \lambda) y_0 + e^{\mathrm{i}(n+1)s}\hat{\Theta}_{n+2}(e^{\mathrm{i}s}, \lambda)y_{n+1} = 0, \quad \nonumber \\
&\Theta_{n+2}(e^{-\mathrm{i}s}, \lambda) y_0 + e^{-\mathrm{i}(n+1)s}\hat{\Theta}_{n+2}(e^{-\mathrm{i}s}, \lambda)y_{n+1} = 0.
\end{align}
Note that if $y_0 = y_{n+1} = 0$ then by (\ref{eq:4_8}) we get $X_n \equiv 0$. 
Therefore, the original equation (\ref{eq:main_t_n}) has a non-trivial solution $X_n$
if, and only if, the determinant of the system of equations (\ref{eq:sys1}) is zero.
This is the form of the required equality (\ref{eq:4_4a}).
\end{proof}

To investigate the asymptotic behavior of the functions $\Theta_n$ and $\hat{\Theta}_n$ when $n \rightarrow \infty$,
we introduce the Toeplitz operator (infinity dimensional) which corresponds to the matrix
$(b_{i-j})_{i,j=0}^\infty$. Let $P:\, L_2(\mathbb{T}) \to H_2(\mathbb{T})$ be the projector defined by
\[
\left( P f \right)(t) = \sum_{j=0}^{\infty} f_j t^j,
\]
where
\[
\sum_{j=-\infty}^{\infty} f_j t^j \ \in \ L_2(\mathbb{T}),
\]
and $H_2(\mathbb(T))$ is the Hardy's famous space.
Then
\begin{equation}
\label{eq:4_9a}
[T(b)f](t) = [Pbf](t)
\end{equation}
called the Toeplitz operator with symbol $b$ (see \cite{BS06}).

Subsequent reasoning is essentially based on the general theory of projection methods (see \cite{Uni}, \cite{GF1974}, \cite{Koz1973}).
Recall the definition of Wiener-Hopf factorization. 
Let the function $f$ belong to the Wiener class, i.e. $f \in W^\alpha$ and $f(t)\ne 0$, $t \in \mathbb{T}$.
Then there is the representation of the function $f$ as the following product:
$$f = f_+ t^{\kappa}f_-,$$
where $\kappa = \mathrm{wind}(f)$, $f_{\pm} \in W_{\pm}^\alpha$, and $wind(f_{\pm}) = 0$.
Here we assume
$$W_\pm^\alpha = \{f \in W^\alpha: f(t) = \sum_{j=0}^{\infty}f_{\pm j} t^{\pm j}\}.$$

By virtue of Lemma \ref{lm:wind_b_n} and equation (\ref{eq:inf_b_n}), a function $b(t,\lambda)$
is factorisable in the space $W^{\alpha-2}$ (see Lemma \ref{lm:b_in_w} i)),
while the factorization factors $\hat{b}_\pm(t,\lambda)$ can be written in the form:
\begin{equation}
\label{eq:4_10}
\hat{b}_{\pm}(t, \lambda) = \exp\left(\frac{1}{2}\log(\hat{b}(t, \lambda)) \pm \frac{1}{2\pi \mathrm{i}} \int_{\mathbb{T}}\dfrac{\log(\hat{b}(\tau, \lambda))}{\tau - t}d\tau\right),
\end{equation}
where the integral is understood in the sense of the principal value.
The functions $\hat{b}_\pm(t,\lambda)$ can be analytically continued inside and outside respectively
of the unit circle $\mathbb{T}$ by the formula:
\[
\hat{b}_{\pm}(t, \lambda) = \exp\left(\pm \frac{1}{\pi \mathrm{i}} \int_{\mathbb{T}}\dfrac{\log(\hat{b}(\tau, \lambda))}{\tau - t}d\tau\right), \qquad |t^{\pm 1}| < 1.
\]

Note that due to the symmetry $b(t^{-1}, \lambda) = b(t, \lambda)$, the factorization factors satisfy the following relation
\begin{equation}
    \label{eq:hat_b_minus}
    \hat{b}_{-}(t,\lambda) =  \hat{b}_{+}(t^{-1},\lambda)/\chi_b(\lambda),
\end{equation}
where the number $\chi_b$ can be calculated by the formula:
\begin{equation}
    \label{eq:chi_b}
    \chi_b(\lambda) = \exp \left\{\frac{1}{\pi} \int_0^\pi \log \hat{b}(e^{\mathrm{i}\varphi},\lambda) d \varphi \right\} .
\end{equation}
Similarly, according to Lemmas \ref{lm:wind_b_n} and (\ref{eq:inf_3}), the function 
$\hat{b}_n(t,\lambda)$
also has the Wiener-Hopf factorization:
\begin{equation}
    \label{eq:hat_b_n_fact}
    \hat{b}_n(t,\lambda) = \hat{b}_{n,+}(t,\lambda) \hat{b}_{n,-}(t,\lambda) .
\end{equation}
Note that the functions $b_{n,\pm}(t,\lambda)$ represent polynomials of degree $(n-1)$
of the variables $t$ and $t^{-1}$ respectively. In addition, due to (\ref{eq:b_n_norm_1}) we have
\begin{equation}
    \label{eq:hat_b_n_pm_norm}
    \| \hat{b}_{n,\pm}(\cdot, \lambda) \|_{\alpha-2} \le \mathrm{const}\, \|a\|_\alpha .
\end{equation}
Next we need to consider the functions $\hat{b}_{n}(t, \lambda)$ and $b_{n,\pm}(t, \lambda)$ in the ring area
\[
K_n =\left\{ z \in \mathbb{C} \vert 1-\frac{C}{n} \le |z| \le 1+\frac{C}{n} \right\},
\qquad C > 0.
\]
(see the definition of $\Pi_n(a)$ in (\ref{p_n_a_set})).
Consider the numbers $r_n^{\pm}=1\pm \frac{C}{n}$, $C > 0$.
\begin{lm}
\label{lm:norm_f_n}
Let the function $f_n(t) \in L_2^{(n)}$ then
\[
\| f_n(r_n^\pm t) \|_\alpha \le \mathrm{const}\, \| f_n(t) \|_\alpha .
\]
\end{lm}
\begin{proof}
Let
\[
f_n(t) =  \sum_{j=-(n-1)}^{n-1} f_j t^j .
\]
Then
\[
\| f_n(r_n^\pm t) \|_\alpha = \sum_{j=-(n-1)}^{n+1} |f_j|\cdot |r_n^\pm|^j (1+|j|)^\alpha .
\]
Since $|j| \le n-1$,
\[
|r_n^\pm|^j = \left( 1 + \frac{C}{n} \right)^j \le \exp\left(\frac{|j|C}{n}\right) \le \exp(C).
\]
Thus
\[
\| f_n(r_n^\pm t) \|_\alpha \le \exp(C) \sum_{j=-(n-1)}^{n+1} |f_j|\cdot  (1+|j|)^\alpha
= \mathrm{const}\, \| f_n(t) \|_\alpha .
\]
\end{proof}
The following statement follows from the above Lemma.
\begin{lm}
\label{lm:sup_hat_b}
Let $a\in \mathrm{CSL}^\alpha$, $\alpha \ge 2 $, then the functions $\hat{b}(r_n^\pm t,\lambda)\in W^{\alpha-2}$, with
\begin{equation}
    \label{eq:sup_hat_b_n}
    \sup_{\lambda \in \mathcal{R}_n(a)} \| \hat{b}_{n}(r_n^{\pm} t, \lambda) \|_{\alpha-2}
    \le \mathrm{const}\, \|a\|_\alpha .
\end{equation}
Besides, $\hat{b}_{n,\pm}(r_n^{\pm} t, \lambda) \in W^{\alpha-2}$ and
\begin{equation}
    \label{eq:sup_norm_hat_b_n_pm}
    \sup_{\lambda \in \mathcal{R}_n(a)} \| \hat{b}_{n,\pm}(r_n^{\pm} t, \lambda) \|_{\alpha-2}
    \le \mathrm{const}\, \|a\|_\alpha .
\end{equation}
\end{lm}
\begin{proof}
According to Lemma \ref{lm:norm_f_n} we get (\ref{eq:sup_hat_b_n}).
Further, by Lemma \ref{lm:d_b_n} we obtain that for sufficiently large $n$,
$\hat{b}_n(r_{n}^{\pm} t, \lambda) \ne 0$ and
$\mathrm{wind}\, \hat{b}_n(r_{n}^{\pm} t, \lambda) = 0$ for all $t\in \mathbb{T}$.
Thus, the functions $\hat{b}_n(r_{n}^{\pm} t, \lambda)$ admit a Wiener-Hopf factorization in the space $W^{\alpha-2}$ and the following inequalities hold
\[
\| \hat{b}_{n,\pm}(r_{n}^{\pm} t, \lambda) \|_{\alpha-2}
\le \mathrm{const}\, \| \hat{b}_{n}(r_{n}^{\pm} t, \lambda) \|_{\alpha-2} .
\]
Applying Lemma \ref{lm:norm_f_n} and inequality (\ref{eq:sup_hat_b_n}) we get (\ref{eq:sup_norm_hat_b_n_pm}).
\end{proof}

Now we are ready to get an asymptotic representation of the functions $\Theta_{n+2}(t,\lambda)$,
from the Theorem \ref{thm:4_2}.
To do this, we note that the inverse of the Toeplitz operator (\ref{eq:4_9a}), calculated by the formula:
\[
T^{-1}(b) = b_{+}^{-1}(t) P b_{-}^{-1}(t) ,
\]
where $b(t) = b_{+}(t) b_{-}(t)$, is a Wiener-Hopf factorization in the space $W^\alpha$ (see (\ref{eq:4_10})).
Thus
\begin{equation}
    \label{eq:t_m1_hat_b_n}
    T^{-1}(\hat{b}_n)\chi_0 = \hat{b}_{n,+}^{-1}(t) P\left( \hat{b}_{n,-}^{-1}(\cdot) \cdot 1 \right)(t)
    = \hat{b}_{n,+}^{-1}(t) .
\end{equation}
\begin{lm}
\label{lm:theta_n_asymp}
Let $a \in \mathsf{CSL}^{\alpha}$, $\alpha \ge 2$. Then the following asymptotic representation holds
\[
\Theta_{n+2}(t,\lambda) = \hat{b}^{-1}_{n,+}(t,\lambda) + \tilde{R}_1^{(n)}(t,\lambda), \quad
\hat{\Theta}_{n+2}(t,\lambda) = \hat{b}^{-1}_{n,-}(t^{-1},\lambda) / \chi_b(\lambda) + \tilde{R}_2^{(n)}(t,\lambda),
\]
where, for $n\to \infty$
\[
\sup\left\{ \left| \tilde{R}_j^{(n)}(z,\lambda) \right|\, : (z,\lambda) \in K_n\times \mathcal{R}_n(a)
 \right\} = o\left( \sfrac{1}{n^{\alpha-2}}\right), \qquad j=1, 2,
\]
and $\chi_b(\lambda)$ is given by (\ref{eq:chi_b}).
\end{lm}

\begin{proof}
From the definition of $\Theta_n(t, \lambda)$ (see (\ref{eq:t_m1_hat_b_n})) it follows that
\[
\Theta_n(t, \lambda) = \hat{b}_{n,+}^{-1}(t, \lambda) + R^{(n)}_1(t, \lambda),
\]
where
\begin{align*}
\tilde{R}_1^{(n)}(t, \lambda) &= T_{n+2}^{-1} \left( \hat{b}_n(\cdot,\lambda) \right) \chi_0
- T^{-1} \left( \hat{b}_n(\cdot,\lambda) \chi_0 \right) \\
&= T_{n+2}^{-1} \left( \hat{b}_n(\cdot,\lambda) \right) \left[ T\left( \hat{b}_n(\cdot,\lambda) \right)
- T_{n+2} \left( \hat{b}_n(\cdot,\lambda) \right) \right]
T^{-1} \left( \hat{b}_n(\cdot,\lambda) \right) \chi_0 - Q_{n+2}T^{-1} \left( \hat{b}_n(\cdot,\lambda) \right)\chi_0.
\end{align*}
By using the obvious equalities
\[
P b P = P_n b P_n + P_n b Q_n + Q_n b P_n + Q_n b Q_n
\]
and (\ref{eq:t_m1_hat_b_n}) we get
\begin{align}
    \tilde{R}_1^{(n)}(t, \lambda) &= \left[ T_{n+2}^{-1} \left( \hat{b}_n(\cdot,\lambda) \right)
    P_{n+2} \hat{b}_n(\cdot,\lambda) Q_{n+2} \hat{b}_n^{-1}(\cdot,\lambda) \right] (t)
    \nonumber
    \\ &- \left[ Q_{n+2} \left( \hat{b}_{n,+}^{-1}(\cdot,\lambda) \right) \right](t),
    \qquad t \in \mathbb{T} .
    \label{eq:tilde_r_1_n}
\end{align}
Thus
\begin{align*}
\| \tilde{R}_1^{(n)}(\cdot, \lambda) \|_{\alpha-2} &\le \left( \| T_{n+2}^{-1} \left( \hat{b}_n(\cdot,\lambda) \right) \|_{\alpha-2}
\cdot \| P_{n+2} \left( \hat{b}_n(\cdot,\lambda) \right) \|_{\alpha-2} + 1 \right) \\
&\times \| Q_{n+2} \left( \hat{b}_{n,+}^{-1}(\cdot,\lambda) \right) \|_{\alpha-2} .
\end{align*}
From (\ref{eq:norm_t_n_m1})
\[
\| T_{n+2}^{-1} \left( \hat{b}_n(\cdot,\lambda) \right) \|_{\alpha-2} \le M .
\]
Further
\[
\| P_{n+2} \left( \hat{b}_n(\cdot,\lambda) \right) \|_{\alpha-2}
\le \mathrm{const}\, \| \hat{b}_n(\cdot,\lambda) \|_{\alpha-2}
\le \mathrm{const}\, \| \hat{b}(\cdot,\lambda)\|_{\alpha-2}.
\]
Finally, Lemma \ref{lm:f_w_alpha} i) gives
\begin{equation}
    \label{eq:norm_q_n_2}
    \| Q_{n+2} \left( \hat{b}_{n,+}^{-1}(\cdot,\lambda) \right)\|_{\alpha - 2}
    = o\left( \sfrac{1}{n^{\alpha-2}} \right) .
\end{equation}
Thus
\begin{equation}
    \label{eq:norm_tilde_r_1_n}
    \| \tilde{R}_1^{(n)} \|_{\alpha-2} = o\left( \sfrac{1}{n^{\alpha-2}} \right) .
\end{equation}
In the above calculations $t \in \mathbb{T}$.
Consider case $z \in K_n$, i.e.
\begin{equation}
    \label{eq:z_r_t}
    z=r t , \qquad r\in [r_n^{-}, r_n^{+}] .
\end{equation}
Denote the first and second term in (\ref{eq:tilde_r_1_n}), respectively, $R^{(n)}_{1,1}(t,\lambda)$
and $R^{(n)}_{1,2}(t,\lambda)$.
Note that $R^{(n)}_{1,1}(t,\lambda) \in L_2^{(n+2)}$ and with
(\ref{eq:norm_q_n_2})--(\ref{eq:norm_tilde_r_1_n}), we get
\[
\| \tilde{R}^{(n)}_{1,1}(t,\lambda) \|_{\alpha-2} = o\left( \sfrac{1}{n^{\alpha-2}} \right),
\qquad t\in \mathbb{T} .
\]
Thus, Lemma \ref{lm:norm_f_n} implies that
\[
\sup_{r \in [r_n^{-}, r_n^{+}]} \| \tilde{R}_{1,1}^{(n)}(r t,\lambda) \|_{\alpha-2}
= o\left( \sfrac{1}{n^{\alpha-2}} \right) .
\]
Consider now $\tilde{R}^{(n)}_{1,2}(t,\lambda) := \left[ Q_{n+2} \left( \hat{b}_{n,+}^{-1}(\cdot,\lambda) \right) \right](t)$.
The function $b_{n,+}(t,\lambda) \in L_2^{(n+2)}$ and by Lemma \ref{lm:sup_hat_b} we have
\[
\sup_{r \in [r_n^{-}, r_n^{+}]} \| b_{n,+}(r t,\lambda) \|_{\alpha-2}
\le \mathrm{const}\, \|a\|_\alpha .
\]
In addition, Lemma \ref{lm:d_b_n} implies that
\[
\sup_{z\in K_n} |b_{n,+}(z,\lambda)| \ge \delta,
\]
where $\delta$ does not depend on $n$ and $\lambda$.
According to a standard theorem of the theory of Banach algebras, for all 
$r \in [1 - C/n, 1 + C/n]$,
$b_{n,+}^{-1}(rt, \lambda) \in W^{\alpha-2}$, and in addition
\begin{align*}
\| b_{n,+}^{-1}(rt, \lambda) \|_{\alpha-2} &\le \mathrm{const}\, \delta^{-1} \| b_{n,+}(rt, \lambda) \|_{\alpha-2} \\
&\le \mathrm{const}\, \delta^{-1} \| a(t,\lambda) \|_{\alpha},
\qquad (|t|=1).
\end{align*}
Thus, according to Lemma \ref{lm:f_w_alpha}  i) it is possible to show that
\[
\sup_{r \in [r_n^{-}, r_n^{+}]} \| \tilde{R}^{(n)}_{1,2}(r t,\lambda) \|_{\alpha-2}
= o\left( \sfrac{1}{n^{\alpha-2}} \right) 
\]
and therefore
\[
\sup_{r \in [r_n^{-}, r_n^{+}]} \| \tilde{R}^{(n)}_{1}(r t,\lambda) \|_{\alpha-2}
= o\left( \sfrac{1}{n^{\alpha-2}} \right) .
\]
Since
\[
\|f\|_0 \le \|f\|_{\alpha-2} ,
\]
this Lemma is proved for the case $j=1$.
The case $j=2$ treated similarly.
\end{proof}

Denote
\begin{align*}
\hat{b}_n(t, g_n(s)) := b_n(t,s), \qquad \hat{b}_{n,\pm}(t, g_n(s)) := b_{b, \pm}(t,s),\\
\tilde{R}_j(t, g_n(s)) := R_j(t, s), \qquad j=1,2.
\end{align*}
Note that as in  (\ref{eq:hat_b_minus}) we have
\begin{equation}
    \label{eq:b_n_m_p}
    b_{n,-}(t^{-1}, s) =  b_{n,+}(t,s)/\chi_b.
\end{equation}
We introduce a continuous function $\eta_n(s)$ satisfying the relation
\begin{equation}
    \label{eq:b_n_m_p_frac}
\frac{b_{n,+}(e^{\mathrm{i}s},s)}{b_{n,+}(e^{-\mathrm{i}s},s)} = e^{-\mathrm{i}\eta_n(s)},
\quad s \in \Pi_n(a),
\end{equation}
We take the continuous branch of the function $\eta_n(s)$ assuming $\eta_n(0) = 0$. It is not difficult to see that
\begin{equation}
    \label{eq:eta_n_0_pi}
    \eta_n(\pi) = \eta_n(0) = 0.
\end{equation}

\begin{lm}
\label{lm:eta_n_pi_j}
Let $a \in \mathsf{CSL}^\alpha$, $\alpha \ge 2$.
Then there is such a large enough natural $N_0$ that
for all $n \ge N_0$ the number $\lambda$ is an eigenvalue of $T_n(a)$
if and only if $j \in \mathbb{Z}$ and $s \in \Pi_n(a)$ satisfy the equation
\begin{equation}
    \label{eq:eta_n_r6}
    (n+1) s + \eta_n(s) + R_6^{(n)}(s) = \pi j,
\end{equation}
$\lambda=g_n(s)$ where $R_6^{(n)}(s)$ satisfies the following asymptotic relation with respect to $n \to \infty$:
\begin{equation}
    \label{eq:r6_estim}
    R_6^{(n)}(s) = o\left( \sfrac{1}{n^{\alpha-2}}\right)
\end{equation}
-- uniformly in parameter $s \in \Pi_n(a)$.
\end{lm}
\begin{proof}
Considering the results of Lemma \ref{lm:theta_n_asymp}, we rewrite equality (\ref{eq:4_4a}) in the form
\begin{align*}
e^{-\mathrm{i}(n+1)s}& \left( b_{n,+}(e^{\mathrm{i}s},s) + R_1^{(n)}(e^{\mathrm{i}s},s)\right)
\left( b_{n,-}(e^{-\mathrm{i}s},s) + R_2^{(n)}(e^{-\mathrm{i}s},s)\right) \\
=& e^{\mathrm{i}(n+1)s} \left( b_{n,+}(e^{-\mathrm{i}s},s) + R_1^{(n)}(e^{-\mathrm{i}s},s)\right)
\left( b_{n,-}(e^{\mathrm{i}s},s) + R_2^{(n)}(e^{\mathrm{i}s},s)\right),
\end{align*}
\[
e^{2 \mathrm{i}(n+1)s} = \frac{b_{n,+}(e^{\mathrm{i}s},s) b_{n,-}(e^{-\mathrm{i}s},s)
\left( 1 + R_3^{(n)}(s) \right)}
{b_{n,+}(e^{-\mathrm{i}s},s) b_{n,-}(e^{\mathrm{i}s},s) \left( 1 + R_4^{(n)}(s) \right)}
\]
where
\begin{align*}
R_3^{(n)}(s) &= b_{n,+}^{-1}(e^{\mathrm{i}s},s) R_1^{(n)}(e^{\mathrm{i}s},s)
+ b_{n,-}^{-1}(e^{-\mathrm{i}s},s) R_2^{(n)}(e^{-\mathrm{i}s},s)\\
&+ b_{n,+}^{-1}(e^{\mathrm{i}s},s) b_{n,-}^{-1}(e^{-\mathrm{i}s},s) R_1^{(n)}(e^{\mathrm{i}s},s)
 R_2^{(n)}(e^{-\mathrm{i}s},s),
\end{align*}
\[
R_4^{(n)}(s) = R_3^{(n)}(-s) .
\]
Considering (\ref{eq:b_n_m_p}), we get
\[
e^{2 \mathrm{i}(n+1)s} = e^{-2 \mathrm{i} \eta_n(s)} \frac{1 + R_3^{(n)}(s)}{1 + R_4^{(n)}(s)} .
\]
Let
\[
R_5^{(n)}(s) := \log\left( \frac{1 + R_3^{(n)}(s)}{1 + R_4^{(n)}(s)} \right)
\]
then
\[
e^{2 \mathrm{i}(n+1)s} = e^{-2 \mathrm{i} \eta_n(s) + R_5^{(n)}(s)} .
\]
The last equation is equivalent to the following set of equations:
\[
2 \mathrm{i} (n+1) s = -2 \mathrm{i} \eta_n(s) + R_5^{(n)}(s) + 2 \mathrm{i} \pi j,
\qquad j \in \mathbb{Z} .
\]
Assuming now
\[
R_6^{(n)}(s) := -\frac{R_5^{(n)}(s)}{2\mathrm{i}}
\]
we get
\[
(n+1) s + \eta_n(s) + R_6^{(n)}(s) = \pi j,
\qquad j \in \mathbb{Z} .
\]
Considering now the relations connecting $R_6^{(n)}(s)$ with $R_1^{(n)}(e^{\pm \mathrm{i}s}, s)$
and $R_2^{(n)}(e^{\pm \mathrm{i}s})$, we get now the asymptotic expansion (\ref{eq:r6_estim}).
\end{proof}

\section{Solvability Analysis of Equation (\ref{eq:eta_n_r6})}

Rewrite the equation (\ref{eq:eta_n_r6}) in the form
\begin{equation}
    \label{eq:f_n_r6}
    F_n(s) + \frac{R_6^{(n)}(s)}{n+1} = d_{j,n} ,
\end{equation}
where
\begin{equation}
    \label{eq:f_n_def}
    F_n(s) := s + \frac{\eta_n(s)}{n+1} ,
\end{equation}
and
\[
d_{j,n} := \frac{\pi j}{n+1} .
\]

Along with (\ref{eq:f_n_r6}), consider the approximating equation

\begin{equation}
    \label{eq:f_n_approx}
    F_n(s) = d_{j,n} , \quad j = 1, 2, \ldots, n .
\end{equation}

We introduce the notion of the modulus of continuity in the complex domain. Let a function $f(z)$ be continuous in some bounded domain $G$ of the complex plane. Then the modulus of continuity $f(z)$ is the function:
\[
w_f(\delta) := \sup_{z_{1,2} \in G,\ |z_1 - z_2| \le \delta} |f(z_1)-f(z_2)|,
\qquad 0<\delta\le \delta_0 .
\]
Let us introduce the domains:
\begin{equation}
    \label{eq:pi_j_n_def}
    \Pi_{j,n}(a) := \left\{ s \in \Pi_n(a)\, : \, |s-e_{j,n}| \le \frac{c_n}{n+1} \right\},
\end{equation}
where
\begin{equation} \label{eq:pi_j_n_def1}
e_{j,n} := d_{j,n} - \frac{\eta_n(d_{j,n})}{n+1} \text{ and }
c_n := 2 \left\| R_6^{(n)}(s) \right\|_{\infty} + w_{\eta_n}\left( \frac{2\|\eta_n\|_\infty}{n+1} \right),
\end{equation}
and the norm  $\|\cdot\|_\infty$ is defined in the standard way on the set  $G$, where $G=\Pi_n(a)$. Recall, that
\[
\Pi_n(a) = \left\{ s=s_0+\mathrm{i}\delta \, \vert \, s_0 \in [c n^{-1}, \pi-c n^{-1}],\ \delta \in [-C n^{-1}, C n^{-1}] \right\},
\]
where $c$, $C$ are some fixed positive numbers such that $c$ is small enough and $C$ is large enough.

The following statement will apply the principle of contractive mappings to the analysis of the solvability of 
(\ref{eq:f_n_r6}), (\ref{eq:f_n_approx}).

Let's introduce the mappings
\[
\Phi_{j,n}(s) := d_{j,n} - \frac{\eta_n(s)}{n+1} .
\]

\begin{lm}
\label{lm:s_princip}
Let the function $a \in \mathsf{CSL}^\alpha$, $\alpha \ge 2$.
Then, if  $s \in \Pi_{j,n}(a)$,
\begin{enumerate}[i)]
    \item $\Phi_{j,n}(s) \in \Pi_{j,n}(a)$ .
    \item $\left( \Phi_{j,n}(s) + \dfrac{R_6^{(n)}(s)}{n+1} \right) \in \Pi_{j,n}(a)$.
\end{enumerate}
\end{lm}
\begin{proof}
We prove ii). Let $s \in \Pi_{j,n}(a)$, then for sufficiently large $n$, we get:
\begin{align*}
& \left| \Phi_{j,n}(s) + \frac{R_6^{(n)}(s)}{n+1} - e_{j,n} \right| \le
\frac{|\eta_n(s) - \eta_n(d_{j,n})|}{n+1} + \frac{|R_6^{(n)}(s)|}{n+1} \\
& \quad \le \frac{w_{\eta_n}(|s - d_{j,n}|)}{n+1} + \frac{\|R_6^{(n)}\|_\infty}{n+1} \\
& \quad \le \dfrac{w_{\eta_n}\left(|s - e_{j,n}| + \frac{|\eta_n(d_{j,n})|}{n+1}\right)}{n+1} + \dfrac{\|R_6^{(n)}\|_\infty}{n+1} \\
& \quad \le \dfrac{w_{\eta_n} \left(\frac{c_n}{n+1} + \frac{\|\eta_n\|_\infty}{n+1} \right)}{n+1} + \frac{\|R_6^{(n)}\|_\infty}{n+1} \\
& \quad \le \dfrac{w_{\eta_n} \left(2 \frac{\|\eta_n\|_\infty}{n+1} \right)}{n+1} + \frac{\|R_6^{(n)}\|_\infty}{n+1} \le \frac{c_n}{n+1} .
\end{align*}
Thus, item ii) of the Lemma is proved.

The item i) is proved similarly if we put $R_6^{(n)}(s) \equiv 0$.
\end{proof}

\begin{thm}
\label{thm:solutions_5_3}
Let the function  $a \in \mathsf{CSL}^\alpha$. Then
\begin{enumerate}[i)]
    \item For $\alpha \ge 2$ the equation (\ref{eq:f_n_r6}) has a unique solution
    $s_{j,n} \in \Pi_{j,n}(a)$, $j=1, 2, \ldots, n$,
    and all  $s_{j,n}$ are different.
    \item Let $s_{j,n}^{*}$ be a solution of the equation (\ref{eq:f_n_approx}) belonging 
    to $\Pi_{j,n}(a)$. Then for $\alpha \ge 3$
    \[ \| s_{j,n} - s_{j,n}^{*} \| = O\left(\sfrac{1}{n^{\alpha-1}}\right) , \]
    where the estimate is uniform in $n$ and $j$.
\end{enumerate}
\end{thm}

\begin{proof}
Let us prove statement i). For this purpose, consider the sequence
\[
s_{j,n}^{(0)} = e_{j,n}, \qquad s_{j,n}^{(k+1)} = \Phi_{j,n}(s_{j,n}^{(k)})
+ \frac{R_6^{(n)}(s_{j,n})}{n+1}, \quad  k=1,2, \ldots .
\]
According to Lemma \ref{lm:s_princip} ii), the sequence $\{s_{j,n}^{(k)}\}_{k=1}^\infty$ is contained in the domain $\Pi_{j,n}(a)$.
Choose from it some convergent subsequence and denote its limit by $\tilde{s}_{j,n}$. Obviously, $\tilde{s}_{j,n}$ satisfies (\ref{eq:f_n_r6}).
Note that for any $j_1 \ne j_2$
\[
\Pi_{j_1,n}(a) \cap \Pi_{j_2,n}(a) = \varnothing ,
\]
because  $|e_{j_1,n} - e_{j_2,n}| \ge \frac{\Delta}{n+1}$, ($\Delta > 0$), while
$\mathrm{diam}\, \Pi_{j,n}=o(1/n)$.
Thus, according to the Lemma \ref{lm:eta_n_pi_j}, the numbers  $g(\tilde{s}_{j,n})$, $j=1, 2, \ldots, n$ are eigenvalues of the matrix $T_n(a)$.
Since this matrix has at most $n$, then $\tilde{s}_{j,n}$ is a unique solution of equation (\ref{eq:f_n_r6}) in the domain $\Pi_{j,n}(a)$. Denoting $\tilde{s}_{j,n} := s_{j,n}$, we completed the proof of i).

Let us prove  ii). Substituting into the equations  (\ref{eq:f_n_r6}) and (\ref{eq:f_n_approx})
respectively, $s_{j,n}$ and  $s_{j,n}^{*}$,  and subtracting the second expression  from the first one, we get
\begin{equation}
    \label{eq:s_j_n_star_estim}
(s_{j,n} - s_{j,n}^{*}) + \frac{\eta_n(s_{j,n}) - \eta_n(s_{j,n}^{*})}{n+1}
= -\frac{R^{(n)}_{6}(s_{j,n})}{n+1} .
\end{equation}
Since $\alpha \ge 3$, according to Lemma \ref{eq:b_n_tilde_t_tau}, $\eta_n(s)$ has a derivative that is bounded uniformly in $n$.
In this way we obtain
\[
|\eta_n(s_{j,n}) - \eta_n(s_{j,n}^{*})| \le \eta_1 |s_{j,n} - s_{j,n}^{*}|,
\]
where
\[
\eta_1 = \sup_{n\in\mathbb{N}} \sup_{s\in \Pi_n(a)} |\eta_n'(s)| < \infty .
\]
From (\ref{eq:s_j_n_star_estim}) we get
\[
|s_{j,n} - s_{j,n}^{*}| \le \frac{\eta_1|s_{j,n} - s_{j,n}^{*}|}{n+1}
+ \frac{\left| R^{(n)}_6(s_{j,n}) \right|}{n+1} .
\]
From (\ref{eq:r6_estim}) it follows that
\[
|s_{j,n} - s_{j,n}^{*}| \left( 1 - \frac{\eta_1}{n+1} \right) \le \frac{\left| R^{(n)}_6(s_{j,n}) \right|}{n+1}
\]
and finally follows
\[
|s_{j,n} - s_{j,n}^{*}| = o\left( \sfrac{1}{n^{\alpha-1}} \right) .
\]
\end{proof}

The statement proved above shows that the roots $s_{j,n}$ of the equation (\ref{eq:f_n_r6}) can be approximated by the roots $s_{j,n}^{*}$ of the equation (\ref{eq:f_n_approx}) for large values of $n$.
Besides, the values $s_{j,n}^{*}$ can be approximated using the method of successive approximations by the values $s_{j,n}^{*(k)}$ defined in the following way:
\begin{equation}
    \label{eq:s_j_n_gamma}
    s_{j,n}^{*(0)} = e_{j,n}, \quad s_{j,n}^{*(k+1)} = \Phi_n(e_{j,n}^{*(k)}),
    \quad k=0, 1, \ldots .
\end{equation} 

\begin{lm}
Let the function $a \in \mathsf{CSL}^\alpha$, $\alpha \ge 3$,
then the equation (\ref{eq:f_n_approx}) has a unique solution $s_{j,n}^{*} \in \Pi_{j,n}(a)$,
$j=1, 2, \ldots, n$ and 
for sufficiently large $n$, the following estimate is valid:
\begin{equation}
    \label{eq:s_j_n_star_3}
    |s_{j,n}^{*} - s_{j,n}^{*(k)}| \le 2 \frac{\eta}{\eta_1^2} \left( \frac{\eta_1}{n+1} \right)^{k+2} ,
\end{equation}
where
\begin{align*}
\eta &= \sup_{n\in\mathbb{N}} \sup_{s\in\Pi_{n}(a)} |\eta_n(s)|, \\
\eta_1 &= \sup_{n\in\mathbb{N}} \sup_{s\in\Pi_{n}(a)} |\eta'_n(s)|.
\end{align*}
\end{lm}

\begin{proof}
We show that sequence $\left\{s_{j,n}^{*(k)}\right\}_{k=1}^\infty$ is convergent.
Indeed, according to Lemma \ref{lm:eta_props3}, the functions $\eta_n(s)$ and $\eta'_n(s)$ are bounded uniformly respect to $n$.
That is, the value
\[
\sup_{n\in\mathbb{N}} \sup_{s\in \Pi_n(a)} |\eta_n'(s)| = \eta_1 < \infty .
\]
Then we have:
\begin{align*}
|s_{j,n}^{*(1)} - s_{j,n}^{*(0)}| = \dfrac{| \eta_n(e_{j,n}) - \eta_{n}(d_{j,n}) |}{n+1} \le
\eta_1 \frac{|e_{j,n} - d_{j,n} |}{n+1} \\
= \eta_1 \dfrac{|\eta_n(d_{j,n})|}{(n+1)^2} \le \frac{\eta_1 \eta}{(n+1)^2}.
\end{align*}
Further
\begin{align*}
|s_{j,n}^{*(2)} - s_{j,n}^{*(1)}| = |\Phi_{j,n}(s_{j,n}^{*(1)}) - \Phi_{j,n}(s_{j,n}^{*(0)})|
= \frac{|\eta_n(s_{j,n}^{*(1)}) - \eta_n(s_{j,n}^{*(0)}) |}{n+1} \\
\le \eta_1 \frac{|s_{j,n}^{*(1)} - s_{j,n}^{*(0)}|}{n+1} \le
 \frac{\eta_1^2 \eta}{(n+1)^3} .
\end{align*}
Similarly
\begin{equation}
    \label{eq:s_j_n_gamma2}
    |s_{j,n}^{*(k+1)} - s_{j,n}^{*(k)}| \le \frac{\eta \eta_1^k }{(n+1)^{k+2}} .
\end{equation}
Since $\sfrac{\eta_1}{n+1} < 1$ for sufficiently large $n$, then the sequence
$\left\{s_{j,n}^{*(k)}\right\}_{k=1}^\infty$ converges to $s_{j,n}^{*}$.

From the estimate (\ref{eq:s_j_n_gamma2}) it follows that
\[
|s_{j,n}^{*(k+m)} - s_{j,n}^{*(k)}| \le \frac{\eta}{\eta_1^2} \left( \frac{\eta_1}{n+1} \right)^{k+2}
\dfrac{1 - \left(\sfrac{\eta_1}{n+1}\right)^{m+1}}{1 - \sfrac{\eta_1}{n+1}} .
\]
Assuming in this inequality that $n+1 > \eta_1$ and passing to the limit when $m\to \infty$, we get that
\[
|s_{j,n}^{*} - s_{j,n}^{*(k)}| \le \frac{\eta}{\eta_1^2} \left( \frac{\eta_1}{n+1} \right)^{k+2}
\frac{1}{1 - \sfrac{\eta_1}{n+1}} .
\]
The assertion of the Lemma obviously follows from this inequality.

\end{proof}

Now we are ready to  prove the main results of the work.

\section{Proof of the main results}

Theorem \ref{thm:a1} follows from Lemma \ref{lm:eta_n_pi_j} and Theorem \ref{thm:solutions_5_3}.

\subsection{Proof of Theorem \ref{thm:sjn1}}

Let $2 \le \alpha < 3$. We estimate the error term
\[
\Delta_{2}^{(n)}(j) := s_{j,n} - e_{j,n} .
\]
We express  $s_{j,n}$ from equation  (\ref{eq:f_n_r6}) and obtain
\begin{align*}
|s_{j,n} - e_{j,n}|  &= \left| \frac{\eta_n(d_{j,n}) - \eta_n(s_{j,n})}{n+1}
+ \frac{R_n^{(6)}(s)}{n+1} \right| \le
\frac{w_{\eta_n}(|s_{j,n} - d_{j,n}|)}{n+1} + \frac{|R_n^{(6)}(s)|}{n+1} \\
& \le  \frac{w_{\eta_n}\left(|s_{j,n} - e_{j,n}| + \frac{\eta_n(d_{j,n})}{n+1}\right)}{n+1}
 + \frac{|R_n^{(6)}(s)|}{n+1}
 \le \frac{w_{\eta_n}\left(\frac{c_n + \|\eta\|_{\infty}}{n+1}\right)}{n+1}
 + \frac{\|R_n^{(6)}\|_\infty}{n+1} ,
\end{align*}
where the value  $c_n$ is given in (\ref{eq:pi_j_n_def1}). Thus, for sufficiently large $n$ we get
\begin{equation}
    \label{eq:sjn_m_djn_estim}
    |s_{j,n} - e_{j,n}| \le \frac{w_{\eta_n}\left(\frac{2 \eta}{n+1}\right)+\|R_n^{(6)}\|_\infty}{n+1} ,
\end{equation}
where  $\eta_n$ is given by formula (\ref{eq:b_n_m_p_frac}).

Let now  $2 < \alpha < 3$. Then, according to Lemma \ref{lm:eta_props3}, $\eta_n \in H^{\alpha-2}(\Pi_n(a))$
with norm bounded uniformly in $n$. Thus 
\begin{equation} \label{eq:6.1}
|s_{j,n} - e_{j,n}| = O\left( \left( \frac{2 \eta}{n+1} \right)^{\alpha-2} \cdot \frac{1}{n+1}
\right) + \frac{o(\sfrac{1}{n^{\alpha-2}})}{n+1} = O\left(\sfrac{1}{n^{\alpha - 1}} \right) ,
\end{equation}
where the estimate is uniform in  $n$.
Let $\alpha=2$. Then $\eta(s) \in C$ and $\eta_n(s)$ has a modulus of continuity with evaluation
uniform in $n$.
Thus, (\ref{eq:sjn_m_djn_estim}) implies
\[
    \Delta_2^{{(n)}}(j)=o\left(\sfrac{1}{n} \right).
\]
And we have that
\begin{equation}
\label{eq:sjn_eq_djn}
s_{j,n} = d_{j,n} - \frac{\eta_n(d_{j,n})}{n+1} + \Delta_2^{{(n)}}(j).
\end{equation}
Thus, from the formulas (\ref{eq:6.1}), (\ref{eq:sjn_eq_djn}) we get the following equality

\[
|\Delta_2^{{(n)}}(j)| =
\begin{cases}
o(1/n), \ \alpha = 2, \\
O\left( n^{-(\alpha-1)} \right),\ 2 < \alpha < 3 .
\end{cases}
\]
Take into account that
\[
\eta_n(d_{j,n}) = o\left(\sfrac{1}{n^{\alpha - 2}} \right)
\]
(\ref{eq:6.1}) give 
Theorem 2 for case $0 \le \alpha < 3$.

Now suppose $3 \le \alpha < 4$. Then consider the difference
\[
\tilde{\Delta}^{(n)}_2(j) := s_{j,n} - s_{j,n}^{*(1)} .
\]
From Theorem \ref{thm:solutions_5_3},  ii)
\[
|s_{j,n} - s_{j,n}^{*}| = o\left( \sfrac{1}{n^{\alpha-1}} \right).
\]
On the other hand, the estimate (\ref{eq:s_j_n_star_3}) gives
\[
|s_{j,n}^{*} - s_{j,n}^{*(1)}| = O\left( \sfrac{1}{n^3}\right).
\]
In this way we have 
\begin{equation}
    \label{eq:delta_n_2_j_estim}
    |\tilde{\Delta}^{(n)}_2(j)| = o\left( \sfrac{1}{n^{\alpha-1}} \right),
\end{equation}
and we can write
\begin{align*}
s_{j,n} &= s_{j,n}^{*(1)} + \tilde{\Delta}^{(n)}_2(j) =
d_{j,n} -\frac{\eta_n(s^{*(0)}_{j,n})}{n+1} + \tilde{\Delta}^{(n)}_2(j) \\
&= d_{j,n}-\dfrac{\eta_n\left( d_{j,n} - \frac{\eta_n(d_{j,n})}{n+1} \right)}{n+1} + \tilde{\Delta}^{(n)}_2(j).
\end{align*}

Since the function $\eta_n(d_{j,n})$ has a derivative, according to Lemma \ref{lm:eta_props3}, with $H^{\alpha-3}(\Pi_n(a))$-norm
bounded uniformly on $n$, then we have that
\[
s_{j_n} = d_{j,n} - \frac{\eta_n(d_{j,n})}{n+1} + \frac{\eta_n'(d_{j,n})\eta_n(d_{j,n})}{(n+1)^2}
+  O\left( \sfrac{1}{n^3} \right) + \tilde{\Delta}_n^{(2)}(j).
\]

Using the Lemma \ref{lm:eta_props3} again and the relation (\ref{eq:delta_n_2_j_estim}), we obtain the statement of the Theorem \ref{thm:sjn1} for the case $3 \le \alpha < 4$.

The case of $\ell \le \alpha < \ell + 1$, $\ell \ge 4$ is treated in a similar way, using the iteration $\varphi_{j,n}^{*(\ell-1)}$ as an approximating expression.

\subsection{Proof of the Theorem \ref{thm:lj1}}

\begin{proof}
From the proved theorem \ref{thm:sjn1} and the definition of the function $g$ we obtain the assertions of the Theorem~\ref{thm:lj1}.
Indeed, since $\lambda^{(n)}_j = g_n(\varphi^{(n)}_j)$, consider the Taylor series decomposition at the point $d_{j,n}$ for the function $g_n$.

We prove formula (\ref{eq:lambda1}) for the first two terms of the expansion in the case $[\alpha]=2$.
As an increment of the argument $\Delta x$ consider the expression
$-\dfrac{\eta(d_{j,n})}{n+1} + \Delta_2^{(n)}(j)$.
According Taylor's formula we have
\[
g_n(x_0+\Delta x) = g_n(x_0) + g_n'(x_0) \Delta x + O(\Delta x^2).
\]
Hence, taking into account the definitions of the functions $g_n$, we obtain the decomposition (\ref{eq:lambda1}):
\begin{align*}
\lambda^{(n)}_j &= g_n(d_{j,n}) - g_n'(d_{j,n}) \left( \frac{\eta(d_{j,n})}{n+1} + \Delta^{(n)}_2(j) \right)
+ O \left( \frac{\eta(d_{j,n})}{n+1} + \Delta^{(n)}_2(j)  \right)^2 \\
&= g_n(d_{j,n}) - \frac{g_n'(d_{j,n}) \eta(d_{j,n})}{n+1} -  g_n'(d_{j,n}) \Delta^{(n)}_2(j)
+ O \left( \frac{\eta(d_{j,n})}{n+1} + \Delta^{(n)}_2(j)  \right)^2 .
\end{align*}
Note now that when $\alpha=2$, 
\[
g_n'(d_{j,n}) = O(d_{j,n}(\pi-d_{j,n})), \qquad \Delta^{(n)}_2(j) = o(1/n),
\qquad \eta(d_{j,n}) = O(d_{j,n}(\pi-d_{j,n})) .
\]
The error term is
\[
o\left( \frac{d_{j,n}(\pi-d_{j,n})}{n} \right) + O \left( \frac{\eta(d_{j,n})}{n+1} + \Delta^{(n)}_2(j)  \right)^2
= o\left( \frac{d_{j,n}(\pi-d_{j,n})}{n} \right).
\]

It now remains to note that for points $\varphi$ lying on the real line, by virtue of Lemma \ref{lm:norm_alpha_n_k}, we have  the equality
$$g_n^{(k)}(\varphi) = g^{(k)}(\varphi) + o(n^{-(\alpha-k)}), \qquad k=0, 1, \ldots , [\alpha].$$
Thus, we obtain that
\[
\lambda^{(n)}_j = g(d_{j,n}) - \frac{g'(d_{j,n}) \eta(d_{j,n})}{n+1}
+ o\left( \frac{d_{j,n}(\pi-d_{j,n})}{n}\right) + o\left( \sfrac{1}{n^2}\right)
+ o\left( \sfrac{1}{n^2}\right) .
\]

We now consider the case of $2 < \alpha < 3$.
Repeating the above reasoning, we obtain the required estimate of the remainder:
\[
O\left( \frac{d_{j,n}(\pi-d_{j,n})}{n^{\alpha-1}} \right).
\]

Consider the case of $[\alpha] = 3$:
\begin{align*}
\lambda^{(n)}_j &=  g_n(d_{j,n}) + g_n'(d_{j,n})\left(- \dfrac{\eta(d_{j,n})}{n+1} + \dfrac{\eta(d_{j,n})\eta'(d_{j,n})}{(n+1)^2} + O \left(\sfrac{1}{n^{\alpha-1}}\right)\right) \\
&\quad + \dfrac{g_n''(d_{j,n})}{2}\left(- \dfrac{\eta(d_{j,n})}{n+1} + \dfrac{\eta(d_{j,n})\eta'(d_{j,n})}{(n+1)^2}
+ O \left(\sfrac{1}{n^{\alpha-1}}\right)\right)^2 + O\left( \frac{d_{j,n}(\pi - d_{j,n})}{n} \right)^3 \\
& =g_n(d_{j,n}) - \frac{g_n'(d_{j,n})\eta(d_{j,n})}{n+1} + 
\frac{\frac{1}{2} g_n''(d_{j,n})\eta^2(d_{j,n}) + g_n'(d_{j,n})\eta(d_{j,n})\eta'(d_{j,n})}{n+1}
+ O\left( \frac{d_{j,n}(\pi-d_{j,n})}{n^{\alpha-1}} \right)
\end{align*}
Also, as above, using Lemma \ref{lm:norm_alpha_n_k}, we obtain that
\begin{align*}
\lambda^{(n)}_j &=  g(d_{j,n}) + g'(d_{j,n})\left(- \dfrac{\eta(d_{j,n})}{n+1} + \dfrac{\eta(d_{j,n})\eta'(d_{j,n})}{(n+1)^2} + O \left(\sfrac{1}{n^{\alpha-1}}\right)\right) \\
&\quad + \dfrac{g''(d_{j,n})}{2}\left(- \dfrac{\eta(d_{j,n})}{n+1} + \dfrac{\eta(d_{j,n})\eta'(d_{j,n})}{(n+1)^2}
+ O \left(\sfrac{1}{n^{\alpha-1}}\right)\right)^2 + O\left( \frac{d_{j,n}(\pi - d_{j,n})}{n} \right)^3 + o\left( \sfrac{1}{n^{\alpha}} \right) \\
& =g(d_{j,n}) - \frac{g'(d_{j,n})\eta(d_{j,n})}{n+1} + 
\frac{\frac{1}{2} g_n''(d_{j,n})\eta^2(d_{j,n}) + g_n'(d_{j,n})\eta(d_{j,n})\eta'(d_{j,n})}{n+1}
+ O\left( \frac{d_{j,n}(\pi-d_{j,n})}{n^{\alpha-1}} \right)
\end{align*}

Thus, we obtain the formula (\ref{eq:lambda1}) and the estimate
\[
\Delta_3^{(n)}(j) = O\left( \frac{d_{j,n}(\pi - d_{j,n})}{n^{\alpha-1}} \right).
\]
The general case is proved similarly.
\end{proof}

\newpage

\bibliographystyle{elsarticle-num}
\bibliography{toeplitz}

\section*{Acknowledgments}
Research of the authors S.M. Grudsky and E.~Ram\'{i}rez de Arellano was supported by CONACYT grant 238630.

Research of the author I.S.~Malisheva was supported by the Ministry of Education and Science of the Russian Federation, Southern Federal University (Project № 1.5169.2017/8.9).

Research of the author V.A.~Stukopin was supported by ``Domestic Research in Moscow'' Interdisciplinary  Scientific Center J.-V. Poncelet (CNRS UMI 2615) and  Scoltech Center for Advanced Study.

\end{document}